\documentclass[11pt]{amsart}

\usepackage{amssymb,amsmath,amscd}
\usepackage{bbm}
\usepackage{a4wide}

\theoremstyle{plain}
\newtheorem{theorem}{Theorem}
\newtheorem{prop}[theorem]{Proposition}
\newtheorem{lemma}[theorem]{Lemma}
\newtheorem{coro}[theorem]{Corollary}

\theoremstyle{definition}
\newtheorem{remark}{Remark}

\newcommand{\ts}{\hspace{0.5pt}}
\newcommand{\nts}{\hspace{-0.5pt}}
\newcommand{\AAA}{\mathbb{A}}
\newcommand{\CC}{\mathbb{C}\ts}
\newcommand{\RR}{\mathbb{R}\ts}
\newcommand{\QQ}{\mathbb{Q}\ts}
\newcommand{\ZZ}{{\ts \mathbb{Z}}}
\newcommand{\SSS}{\mathbb{S}}
\newcommand{\TT}{\mathbb{T}}
\newcommand{\NN}{\mathbb{N}}
\newcommand{\XX}{\mathbb{X}}
\newcommand{\YY}{\mathbb{Y}}

\newcommand{\cA}{\mathcal{A}}
\newcommand{\cB}{\mathcal{B}}

\newcommand{\cD}{\mathcal{D}}
\newcommand{\cH}{\mathcal{H}}

\newcommand{\cP}{\mathcal{P}}
\newcommand{\cX}{\mathcal{X}}
\newcommand{\vL}{\varLambda}

\newcommand{\ii}{\mathrm{i}}
\newcommand{\dd}{\, \mathrm{d}}

\newtheorem{definition}{Definition}

\newcommand{\MM}{\mathcal{M}(G)}
\newcommand{\MCV}{\mathcal{M}_{C,V}(G)}
\newcommand{\MTB}{\mathcal{M}^{\infty}(G)}
\newcommand{\cMO}{\mathcal{M}(\XX)}
\newcommand{\cMT}{\mathcal{M}(\YY)}
\newcommand{\cPO}{\mathcal{P}(\XX)}
\newcommand{\cPGO}{\mathcal{P}_G (\XX)}

\newcommand{\Oomega}{(\XX,\alpha)}

\newcommand{\Ttheta}{(\YY,\beta)}

\newcommand{\LO}{L^2 (\XX,\mu)}
\newcommand{\LT}{L^2 (\YY,\nu)}

\newcommand{\cV}{\mathcal{V}}

\begin{document}

\title[Dynamical versus diffraction spectrum]
{Dynamical versus diffraction spectrum \\[2mm]
for structures with finite local complexity}

\author{Michael Baake}
\address{Fakult\"at f\"ur Mathematik, Universit\"at Bielefeld, \newline
\hspace*{\parindent}Postfach 100131, 33501 Bielefeld, Germany}
\email{mbaake@math.uni-bielefeld.de}

\author{Daniel Lenz}
\address{Fakult\"at f\"ur Mathematik, Universit\"at Jena, \newline
\hspace*{\parindent}Ernst-Abbe-Platz 2, 07743 Jena}
\email{daniel.lenz@uni-jena.de}

\author{Aernout van Enter}
\address{Johann Bernoulli Institute for Mathematics and
  Computer Science, \newline
 \hspace*{\parindent}University of Groningen,
  PO Box 407, 9700\ts AK Groningen, The Netherlands}
\email{a.c.d.van.enter@rug.nl}

\begin{abstract}
  It is well-known that the dynamical spectrum of an ergodic measure
  dynamical system is related to the diffraction measure of a typical
  element of the system. This situation includes ergodic subshifts
  from symbolic dynamics as well as ergodic Delone dynamical systems,
  both via suitable embeddings. The connection is rather well
  understood when the spectrum is pure point, where the two spectral
  notions are essentially equivalent. In general, however, the
  dynamical spectrum is richer.

  Here, we consider (uniquely) ergodic systems of finite local
  complexity and establish the equivalence of the dynamical spectrum
  with a collection of diffraction spectra of the system and certain
  factors. This equivalence gives access to the dynamical spectrum via
  these diffraction spectra. It is particularly useful as the
  diffraction spectra are often simpler to determine and, in many
  cases, only very few of them need to be calculated.
\end{abstract}

\maketitle

\section{Introduction}

Dynamical systems as defined by the translation action of locally
compact Abelian groups (LCAGs) form an important class of structures
whose classification is only partially known.  An important tool is
the dynamical spectrum, which was introduced in \cite{Koopman} and
then largely developed by von Neumann \cite{vNe}. It was used by
Halmos and von Neumann \cite{HvN} to achieve the classification of
ergodic systems with pure point dynamical spectrum up to metric
isomorphism, together with giving canonical representatives in terms
of group additions on compact Abelian groups; see \cite{CFS,EW} for
further details.

In the more general case of a system with mixed dynamical spectrum,
much less is known, although these spectra are practically relevant;
compare \cite{Wit} and references therein for recent examples, and
\cite{BG-review,BBM} for some theoretical counterpart.  The maximal
equicontinuous factor, also known as the Kronecker factor, is a
natural object for analysing the pure point part of the spectrum, but
it is totally blind to continuous spectral components.  In many
concrete examples, it seems advantageous to drop the demand of
equicontinuity and search for a maximal factor with pure point
spectrum, preferably of the same type. This will be a (generally not
one-to-one) cover of the Kronecker factor; see \cite{BGG} and
references therein for some recent results from tiling theory.  It is
known, however, that this approach is not always possible
\cite{Herning}, while it is very efficient when it works; see
\cite{GTM,BGG,squiral} for examples.

A different object, of physical origin and seemingly unrelated at
first sight, is the \emph{diffraction measure} $\widehat{\gamma}$ of a
translation bounded measure $\omega$ on an LCAG $G$. Here, $\omega$
may be viewed as a model of an individual many-particle configuration,
which we assume to be a typical representative of an (ergodic)
ensemble of such structures, so that all quantities under
consideration are well-defined.  Then, the measure $\widehat{\gamma}$
is the Fourier transform of the positive definite autocorrelation
measure $\gamma$ of $\omega$. Interestingly, a closer inspection shows
a deep connection between the diffraction measure of $\omega$ and the
dynamical spectrum of the orbit closure of $\omega$ under the
translation action of $G$.  This connection was exploited in
\cite{Dworkin} and led to the important equivalence theorem between
pure pointedness of dynamical and diffraction spectra for measure
dynamical systems \cite{LMS,Martin,BL,BLM,LS,LM-2}. The relevance of
measure dynamical systems stems from the fact that many dynamical
systems, such as subshifts from symbolic dynamics or Delone dynamical
systems, are naturally embedded into this class of dynamical
systems; compare \cite{BM,BL,LS}.

It was noticed quite early \cite{EM} that a similarly simple
correspondence cannot hold for systems with singular continuous
spectral components, and it was later shown that the same (negative)
conclusion is generally also true in the presence of absolutely
continuous parts \cite{BvE}. In these systems beyond the pure point
case, the dynamical spectrum is richer than the diffraction spectrum
(which is the group generated by the Fourier--Bohr spectrum of
$\gamma$ when $\widehat{\gamma}$ is a pure point measure; see
Eq.~\eqref{eq:FB} below for a precise definition). A main insight of
\cite{BvE} is the observation that, in the examples appearing in
\cite{EM,BvE} as well as in many other ones (compare \cite{TAO} and
references therein), the missing parts of the dynamical spectrum could
be reconstructed from the diffraction measures of suitable factors of
the original system.

The importance of factors is perhaps not surprising, for instance in
the light of Fraczek's theorem \cite{Frac} which asserts that, under
some mild assumption, the maximal spectral measure can be realised as
that of a continuous function; see also \cite{A}. This continuous
function gives rise to a factor where the correspondence between the
dynamical and the diffraction spectrum can be understood via a minor
variant of Dworkin's argument \cite{Dworkin}; see also
\cite{DM}. However, the factor obtained this way might have a rather
complicated structure, as it generally cannot be obtained from a
function of finite range; compare \cite{Herning}.  In particular, in
the case of symbolic dynamics, such a factor will generally not be
realised over a finite alphabet, but rather over the unit disc.  The
observation mentioned above indicates that there might be an
alternative path via a collection of factors, but then significantly
simpler ones.

Given this situation, it is a natural conjecture that, under
reasonable assumptions on the type of the dynamical system, the
dynamical spectrum is equivalent to the \emph{collection} (or union)
of diffraction spectra of the system and its factors, where the latter
should be of a similar kind as the system itself (or simpler). This
conjecture is also supported by the physical intuition that the
autocorrelation essentially is a $2$-point correlation, while
higher-order correlations may still contain important information on
the system. Many of these correlations are not seen by the diffraction
measure of the original system itself, but at least the generalised
$2$-point correlations (between the positions of local patterns, say)
should be accessible via suitable factors and their diffraction
measures.  Since all correlation functions together determine the
entire system (again under suitable assumptions; see \cite{LM}), the
above conjecture is plausible.  The present paper is centred around
this conjecture.

For systems with pure point spectrum, little new insight seems gained
at a first glance, as factors of such systems are pure point again
\cite{BL-2}. Also, as mentioned before, we have equivalence of pure
point diffraction and dynamical spectrum in those cases anyhow; see
\cite{BL,LS} and references therein. Still, as we shall see later,
factors can shed some light on the structure of extinctions.  In other
examples, however, even simple factors may reveal coherent order, such
as the period doubling chain (pure point) for the Thue--Morse chain
(with singular spectrum of mixed type \cite{Q,EM}). Here, the full
dynamical spectrum can be reconstructed from the diffraction measure
of the Thue--Morse chain and this one factor. In particular, one can
represent the maximal spectral measure this way, which is implicit
already in \cite{Q}.  Several other examples are treated in
\cite{TAO}; see also \cite{GTM,BGG,squiral} and references therein.

Below, we make the conjecture precise, and prove it for (uniquely
ergodic) systems of finite local complexity (FLC), which includes
symbolic dynamics on finite alphabets as well as FLC Delone dynamical
systems. After some preliminary material on notions and methods (in
Section~\ref{sec:prelim}), we treat those cases explicitly in one
section each.  While we focus on dynamical systems in $\RR^{d}$ in
those sections, a general abstract approach for the action of LCAGs is
presented in Section~\ref{sec:general}, which also opens a path to
drop the ergodicity and FLC assumptions. This is followed by some
concluding remarks.

\section{Terminology and background}\label{sec:prelim}

Consider a (possibly unbounded) measure $\omega$ on $\RR^{d}$, by
which we mean a continuous linear functional on the space
$C_{\mathsf{c}} (\RR^{d})$ of continuous functions with compact
support. The corresponding weak-$*$ topology is called the \emph{vague
  topology}. Due to the Riesz--Markov theorem, these measures can be
identified with the regular Borel measures on $\RR^{d}$. A measure
$\omega$ is called \emph{translation bounded} when $\sup_{t\in\RR^{d}}
\lvert \omega \rvert (t+K) < \infty$ holds for any compact
$K\subset \RR^{d}$; see \cite{BF,Hof,Martin,TAO} for background
material. Given $\omega$, the measure $\widetilde{\omega}$ is
defined by $\widetilde{\omega} (g) = \overline{\omega(\widetilde{g})}$
for $g\in C_{\mathsf{c}} (\RR^{d})$, with $\widetilde{g} (x) :=
\overline{g (-x)}$.

Given a (translation bounded) measure $\omega$ on $\RR^{d}$, its
\emph{autocorrelation measure} $\gamma^{}_{\omega}$, or
autocorrelation for short, is defined as
\begin{equation} \label{eq:def-auto}
     \gamma^{}_{\omega} \, := \, \omega \circledast \widetilde{\omega}
     \, := \lim_{r\to\infty} \frac{\omega|^{}_{r} * \widetilde{\omega|^{}_{r}}}
     {\mathrm{vol} (B_{r} (0))} \, ,
\end{equation}
where $\omega|^{}_{r}$ denotes the restriction of $\omega$ to the open
ball $B_{r} (0)$, and the limit is assumed to exist (no other
situation will be considered below). The volume-weighted convolution
$\circledast$ of two unbounded measures is sometimes referred to as
the \emph{Eberlein convolution}.  Note that the autocorrelation is
often called \emph{Patterson function} in crystallography
\cite{Cowley}, even though it is a measure in our setting. This
approach was introduced in \cite{Hof}; see \cite[Ch.~9]{TAO} for a
comprehensive exposition and \cite{BG-review} for an informal
summary. Since $\gamma^{}_{\omega}$ is positive definite by
construction, its Fourier transform $\widehat{\gamma^{}_{\omega}}$
exists \cite{BF} and is a positive measure. The latter is called
the \emph{diffraction measure} of $\omega$, which can be seen as
the generalisation of the \emph{structure factor} from
classical crystallography \cite{Cowley}.

Let us expand on the terminology around spectra by means of some
additional definitions.  A (translation-bounded) measure $\omega$ on
$\RR^{d}$ whose autocorrelation $\gamma = \omega \circledast
\widetilde{\omega}$ exists is called \emph{pure point diffractive}
when $\widehat{\gamma}$ is a pure point measure. In this case, the
supporting set
\begin{equation}\label{eq:FB}
   S^{}_{\mathrm{FB}} \, := \,
  \{k\in\RR^{d} \mid \widehat{\gamma} (\{ k\}) > 0 \}
\end{equation}
is known as the \emph{Fourier--Bohr spectrum} of $\gamma$. The set
$S^{}_{\mathrm{FB}}$ is also known as the set of Bragg peak locations
in the physics literature. It is (at most) a countable set, but might
(and generally will) be a dense subset of $\RR^{d}$. Note that
$S^{}_{\mathrm{FB}}$ need not be a group, due to the possibility of
extinctions \cite{Cowley,LM}.

Let $\omega$ be a translation bounded measure and consider $\XX :=
\overline{ \{ \delta^{}_{t} * \omega \mid t \in \RR^{d} \} }$, with
the closure taken in the vague topology. This defines a compact space
that gives rise to a measure-theoretic dynamical system
$(\XX,\RR^{d},\mu)$, with the translation action of $\RR^{d}$ and some
invariant measure $\mu$. The notion of the \emph{dynamical spectrum}
now emerges via the natural unitary (translation) action of $\RR^{d}$
on the Hilbert space $L^{2} (\XX,\mu)$; see \cite{CFS,Nad,Q} for
general background and \cite[App.~B]{TAO} for a brief summary. When
$L^{2} (\XX,\mu)$ possesses a basis of eigenfunctions for the
$\RR^{d}$-action, one speaks of a system with \emph{pure point
  dynamical spectrum}. Then, the set of eigenvalues forms a subgroup
of $\RR^{d}$, known as the \emph{pure point spectrum}. We are thus not
using the term `spectrum' in the sense of the topological spectrum
(which is closed), but in the sense of the set of eigenvalues (which
need not be closed as a set). More generally, when the eigenfunctions
are not total in $L^{2} (\XX,\mu)$, the group of eigenvalues
constitutes the pure point part of the dynamical spectrum, where the
spectral measures attached to the eigenfunctions all are pure point
measures. In particular, $(\XX,\RR^{d},\mu)$ has pure point spectrum
if and only if all spectral measures are pure point.

When $\omega$ is a pure point diffractive measure and $\mu$ is
ergodic, the dynamical spectrum of $(\XX,\RR^{d},\mu)$ is pure point
and can be characterised as the smallest subgroup of $\RR^{d}$ that
contains the supporting set $S^{}_{\mathrm{FB}}$ from
Eq.~\eqref{eq:FB}. We shall say more about this later; see also the
Appendix.  More generally, the positive diffraction measure
$\widehat{\gamma}$ has the unique decomposition
\[
    \widehat{\gamma} \, = \, \bigl(\widehat{\gamma}\bigr)_{\mathsf{pp}}
      + \bigl(\widehat{\gamma}\bigr)_{\mathsf{sc}}
      + \bigl(\widehat{\gamma}\bigr)_{\mathsf{ac}}
\]
into its pure point, singular continuous and absolutely continuous
components. Then, the Fourier--Bohr spectrum is the supporting set of
$\bigl(\widehat{\gamma}\bigr)_{\mathsf{pp}}$. As before, this set is a
countable (and possibly dense) subset of $\RR^{d}$.  In this more
general case, the dynamical spectrum is usually described via the
spectral decomposition theorem for unitary operators, hence via a
suitable collection of spectral measures of (preferably continuous)
functions on $\XX$, and with special emphasis on the spectral measure
of maximal type; compare \cite{Q} for a concise summary. This is
precisely the point of view we will be using below, in the sense that
we will relate the spectral measures of $(\XX,\RR^{d},\mu)$ with the
diffraction measure $\widehat{\gamma}$ of the system and its factors
of the same kind (to be made precise later). Further tools and methods
will be introduced while we proceed.

\section{The case of symbolic dynamics}
\label{sec:symbolic}

Let us begin with the simpler case of symbolic dynamics; see
\cite{LM-Book} for background. Recall that the full shift space
$\cA^{\ZZ}$ over a finite alphabet $\cA$ is compact in the usual
product topology. The latter is also known as the \emph{local
  topology}, because two elements $u,v \in \cA^{\ZZ}$ are close when
$u$ and $v$ agree on a large index range around $0$ (this defines both
a uniform structure and a metrisable topology).  For $u\in \cA^{\ZZ}$,
we write $u = (u_{n})^{}_{n\in\ZZ}$ and use $u^{}_{[m,n]} = u^{}_{m}
u^{}_{m+1} \ldots u^{}_{n}$, with $n\ge m$, for the finite subword
ranging from $m$ to $n$. In particular, $u^{}_{[m,m]} = u^{}_{m}$. The
shift $S$ acts on $\cA^{\ZZ}$ via $(Su)^{}_{n} := u^{}_{n+1}$, which
is continuous and invertible. In particular, $S$ induces a group
action by $\ZZ$, so that $(\cA^{\ZZ},\ZZ)$ is a topological dynamical
system.

Consider now a closed shift-invariant subset $\XX\subset\cA^{\ZZ}$,
which is then compact and known as a subshift, with the additional
property that $\XX$ admits only one shift-invariant probability
measure $\mu$. In other words, we assume that $(\XX,\ZZ,\mu)$ is a
measure-theoretic dynamical system which is uniquely ergodic. If
$\cA^{*}_{\XX}$ denotes the dictionary of $\XX$, by which we mean the
set of all finite words that occur as subwords of some element of
$\XX$, we know from Oxtoby's theorem \cite{Ox,Q} that unique
ergodicity is equivalent to the uniform existence, in any element of
$\XX$, of all frequencies of the words from $\cA^{*}_{\XX}$. All such
frequencies are strictly positive precisely when $\XX$ is also
minimal.  The frequency $\nu^{}_{w}$ of a non-empty word
$w\in\cA^{*}_{\XX}$ defines the measure of any of the corresponding
cylinder sets via $\mu \bigl( \{ x \in \XX \mid x^{}_{[m,m+ \vert w
  \rvert -1]} = w \} \bigr) = \nu^{}_{w}$, where $m\in\ZZ$ is arbitrary
and $\lvert w \rvert$ is the length of $w$. More complicated word
patterns are realised by suitable unions and intersections of
(elementary) cylinder sets. By construction and standard arguments,
this consistently defines a shift-invariant probability measure $\mu$
on $\XX$; see \cite{LM-Book}.

With $\mu$, one also has the Hilbert space $\cH = L^{2} (\XX,\mu)$,
with scalar product
\[
     \langle g \ts | h \rangle \, =  \int_{\XX}
     \overline{g(x)} \, h(x) \dd \mu (x) \ts ,
\]
written here in a way that is linear in the second argument. The shift
$S$ induces a unitary operator $U$ on $\cH$ via $Uf := f \circ S$, so
that $\bigl(Uf\bigr) (x) = f (Sx)$ for all $x\in\XX$.  Since $\XX$ is
compact, the continuous functions $C (\XX)$ are dense in $L^{2}
(\XX,\mu)$ by standard arguments \cite[Ch.~VII.5]{Lang}. The
characteristic function specified by a finite word $w\in\cA^{*}_{\XX}$
together with an index $n\in\ZZ$ is defined by $1^{}_{w,n} (x) = 1$
when $x^{}_{[n,n+\lvert w \rvert- 1]} = w$, and by $1^{}_{w,n} (x) = 0$
otherwise. Any such function is continuous, and all of them together
generate an algebra $\AAA (\XX)$ (under addition and multiplication)
that is dense in $C(\XX)$ by the Stone{\ts}--Weierstrass theorem
\cite[Thm.~III.1.4]{Lang}. It is not hard to see that
\begin{equation}\label{eq:A-def}
   \AAA (\XX) \, = \, \{ f\in C (\XX) \mid
   \text{ $f$ takes only finitely many values}\} \ts ,
\end{equation}
which provides an explicit characterisation.  Indeed, the inclusion
$\subset$ is obvious; the reverse inclusion $\supset$ follows because
any continuous function on $\XX$ with finitely many values is
determined by a finite `window' around $0$.

\smallskip

Given an arbitrary function $f\in \cH$, the map $n \mapsto \langle f
\ts | \ts U^{n} f \rangle$ defines a complex-valued, positive definite
function on the discrete group $\ZZ$, so that, by the
Herglotz--Bochner theorem \cite[Thm.~1.4.3]{Rudin}, there is a unique
positive measure $\sigma = \sigma^{}_{\! f}$ on the dual group
$\SSS^{1}$ (which is identified with the $1$-torus $\TT = \RR/\ZZ$
here) such that
\[
       \langle f \ts | \ts U^{n} f \rangle \, =
      \int_{0}^{1}  e^{2 \pi \ii n t} \dd \sigma^{}_{\! f} (t) \, .
\]
This measure $\sigma^{}_{\! f}$ is known as the \emph{spectral
  measure} of the function $f$.

Consider now an arbitrary, but fixed element $g \in \AAA (\XX)$
subject to the additional requirement that it takes values in $\{ 0,1
\}$ only. As $g\in \AAA (\XX)$, the value $g(x)$ is determined from a
finite index range, the latter being independent of $x\in\XX$.  Define
now the \emph{sliding block map} $\varPhi_{g} \! : \, \XX
\longrightarrow \{ 0,1 \}^{\ZZ}$ via $\bigl(\varPhi_{g} (x)\bigr) (n)
= g( S^{n} x)$ for $x\in\XX$ and $n\in\ZZ$; see \cite{LM-Book} for
background. Clearly, $\varPhi_{g}$ is a continuous map, wherefore $\YY
:= \varPhi_{g} (\XX) \subset \{ 0,1 \}^{\ZZ}$ is compact. Since the
diagram
\[
  \begin{CD}
    \XX @>S>> \XX  \\
    @V{\varPhi_{g}}VV     @VV{\varPhi_{g}}V \\
    \YY_{\phantom{\varrho}} @>S>> \YY
  \end{CD}
\]
is commutative, $(\YY,\ZZ)$ is a topological factor of $(\XX,\ZZ)$.
Moreover, $\mu$ induces a shift-invariant measure $\mu^{}_{\YY}$ on
$\YY$ via $\mu^{}_{\YY} (B) = \mu \bigl( \varPhi^{-1}_{g} (B)\bigr)$
for arbitrary Borel sets $B\subset \YY$. By an application of
\cite[Prop.~3.11]{DGS}, we know that $(\YY,\ZZ,\mu^{}_{\YY})$ is again
uniquely ergodic.

Let $x\in\XX$ be arbitrary but fixed, with $y=\varPhi_{g} (x) \in
\YY$, and consider the corresponding Dirac comb $\omega =
\sum_{n\in\ZZ} y_{n} \ts \delta_{n}$, which is a translation bounded
measure on $\RR$. It possesses the autocorrelation measure
$\gamma^{}_{\omega} = \omega \circledast \widetilde{\omega} = \eta \ts
\delta^{}_{\ZZ}$ with the coefficients
\begin{equation}\label{eq:eta-def}
   \eta (m) \, =  \lim_{N\to\infty} \frac{1}{2N+1}
   \sum_{n=-N}^{N} y_{n} \ts \overline{y_{n-m}} \, =
   \lim_{N\to\infty} \frac{1}{2N+1}  \sum_{n=-N}^{N}
   \overline{y_{n}}\ts y_{n+m} \ts ,
\end{equation}
which are written in the general form that also applies to complex
sequences (even though they are real here).  All limits exist due to
the unique ergodicity of $(\YY,\ZZ,\mu^{}_{\YY})$, wherefore we can
employ the stronger version of the ergodic theorem for the orbit
average of a continuous function (for instance the one defined by $y
\mapsto \overline{y^{}_{0} y^{}_{m}}\ts $), and $\gamma^{}_{\omega}$
is a positive definite measure. Its Fourier transform
$\widehat{\gamma^{}_{\omega}}$ thus exists, and is a positive measure
of the form $\widehat{\gamma^{}_{\omega}} = \varrho *
\delta^{}_{\ZZ}$, with $\varrho =
\widehat{\gamma^{}_{\omega}}\big|_{[0,1)}$. Equivalently, $\eta \! :
\, \ZZ \longrightarrow \RR$ is a positive definite function on $\ZZ$,
see \cite[Lemma~8.4]{TAO}, with representation $\eta(m) = \int_{0}^{1}
e^{2\pi \ii m t} \dd \varrho (t)$, where $\varrho$ is now interpreted
as a positive measure on the $1$-torus $\TT$.

Observe next that $y_{m} = g (S^{m} x)$, wherefore the coefficient
$\eta (m)$ can also be expressed as
\begin{equation}\label{eq:eta-formula}
   \eta(m) \, = \, \lim_{N\to\infty} \frac{1}{2N+1}
   \sum_{n=-N}^{N} \overline{g(S^{n} x)} \, g( S^{n+m} x)
   \, = \, \int_{\XX} \overline{g (x)}\, g(S^{m}x) \dd \mu (x)
   \, = \, \langle g | U^{m} g \rangle ,
\end{equation}
where the second equality is a consequence of unique ergodicity.  This
shows that $\varrho = \sigma^{}_{\! g}$, with $\sigma^{}_{\! g}$ the
spectral measure of $g$. In other words, the spectral measure of the
function $g$ occurs as the `building block' of the diffraction measure
of the factor that is defined via $\varPhi_{g}$. After this explicit,
but somewhat informal, introduction we can now develop the more
general structure.
\smallskip

Let $\XX \subset \cA^{\ZZ}$ be a subshift over the finite alphabet
$\cA$, and let $\cB$ be a finite set (equipped with the discrete
topology).  Then, any continuous $g \! : \, \XX \longrightarrow \cB$
gives rise to a continuous map $\varPhi_g \! :\, \XX \longrightarrow
\cB^{\ZZ}$, defined by $\bigl(\varPhi_g (x)\bigr) (n) := g (S^n x)$,
so that $\YY := \varPhi_g (\XX)$ is a factor of $\XX$. Moreover, any
subshift factor of $\XX$ over $\cB$ arises in this manner. This is a
variant of the Curtis--Lyndon--Hedlund theorem, compare
\cite[Thm.~6.2.9]{LM-Book}, which we formulate in our context as
follows.

\begin{lemma}\label{lem:factors}
  Let\/ $\XX$ be a subshift over the finite alphabet\/ $\cA$ and let\/
  $\YY$ be a subshift over the finite set\/ $\cB$ that is a factor
  of\/ $\XX$, with factor map\/ $\varPhi \! : \, \XX \longrightarrow
  \YY$.  Then, $\varPhi = \varPhi_g$ for\/ $g:= \delta \circ\varPhi$,
  where\/ $\delta \!  : \, \YY \longrightarrow \cB$ is defined by\/ $
  y \mapsto y(0)$ and where\/ $g$ is a continuous function that takes
  only finitely many values.
\end{lemma}
\begin{proof}
  Since $\varPhi = \varPhi_{g}$ is a factor map, we can calculate
\[
  \bigl(\varPhi (x)\bigr) (n) \, = \, g (S^{n} x)
  \, = \, \bigl( \varPhi ( S^n x) \bigr) (0)
  \, = \, \bigl(S^n \varPhi (x) \bigr) (0)
  \, = \, \bigl( \varPhi (x) \bigr) (n) \, ,
\]
which implies the first claim. Since $g$ is continuous on $\XX$ by
construction, but takes only finitely many distinct values,
$\varPhi=\varPhi_{g}$ is indeed a sliding block map.
\end{proof}

This shows that subshift factors over finite sets are in one-to-one
correspondence to continuous functions that take finitely many values.
\smallskip

In view of the connection with diffraction, we now realise the
alphabet as a finite subset of $\CC$.  Let $\XX$ be a uniquely ergodic
subshift over the finite set $\cA \subset \CC$. Then, $\XX$ gives rise
to a canonical autocorrelation $\gamma = \gamma^{}_{\XX}$ as
follows. Consider the Dirac comb $\omega = \sum_n x_n \delta_n$ for an
arbitrary $x\in \XX$. Due to unique ergodicity, the associated
autocorrelation does not depend on $x$, hence effectively only on
$\XX$.  It is this observation that will later pave the way to a more
general (and abstract) approach.  Note that $\gamma^{}_{\XX}$ is a
positive definite measure of the form $\gamma^{}_{\XX} = \eta^{}_{\XX}
\ts \delta^{}_{\ZZ} := \sum_{m\in\ZZ} \eta^{}_{\XX}(m) \ts
\delta_{m}$, where positive definiteness of $\gamma^{}_{\XX}$ as a
measure on $\RR$ is equivalent to that of the function $\eta^{}_{\XX}
\! : \, \ZZ \longrightarrow \CC$; see \cite[Lemma~8.4]{TAO}. The
Fourier transform $\widehat{\gamma^{}_{\XX}}$ of $\gamma^{}_\XX$,
which exists by general arguments \cite{BF}, is a $1$-periodic measure
on $\RR$, as follows from \cite[Thm.~10.3]{TAO}; see also
\cite{Baa}. This gives
\[
   \widehat{\gamma^{}_{\XX}} \, = \,
     \varrho^{}_{\XX} * \delta^{}_{\ZZ} \ts ,
\]
with a finite positive measure $\varrho^{}_{\XX}$. The latter is not
unique in the sense that different $\varrho^{}_{\XX}$ can lead to the
same measure $\widehat{\gamma^{}_{\XX}}$.  A canonical choice is
$\varrho^{}_{\XX} = \widehat{\gamma^{}_{\XX}} \big|_{[0,1)}$, which is
based on the natural fundamental domain $\TT \simeq [0,1)$ of a
$\ZZ$-periodic structure.  This particular choice permits the
simultaneous interpretation of $\varrho^{}_{\XX}$ as a positive
measure on $\TT $, so that
\begin{equation}\label{eq:gen-eta-coeff}
     \eta^{}_{\XX} (m) \, =
     \int_{0}^{1} e^{2 \pi \ii mt} \dd \varrho^{}_{\XX} (t) \ts ,
\end{equation}
in line with the Herglotz--Bochner theorem.  We thus call
$\varrho^{}_{\XX}$ the \emph{fundamental diffraction} of
the subshift $\XX$.

\begin{prop}\label{prop:specmeas}
  Let\/ $\XX$ be a uniquely ergodic subshift over the finite
  alphabet\/ $\cA$. Let\/ $\cB\subset \CC$ be finite and\/ $g \! : \,
  \XX \longrightarrow \cB$ continuous, with spectral
  measure\/ $\sigma^{}_{\! g}$, and let\/ $\YY$ denote the
  associated subshift factor.  Then, the fundamental diffraction of\/
  $\YY$ satisfies\/ $\varrho^{}_{\YY} = \sigma^{}_{\! g}$.
\end{prop}
\begin{proof} This follows exactly as in our previous derivation
around Eqs.~\eqref{eq:eta-def} and \eqref{eq:eta-formula}.
\end{proof}

Note that subsets of $\CC$ are natural objects in the context of
mathematical diffraction theory; see \cite[Ch.~9]{TAO} for
background. Subsets of $\RR$, $\QQ$ or $\ZZ$ are special cases and
also of interest. They are covered by Proposition~\ref{prop:specmeas}
as well. \smallskip

Let us now establish a link between the canonical shift-invariant
measure $\mu$ of $\XX$ (defined via its values on cylinder sets) and
the diffraction measures of the subshift factors.

\begin{prop}\label{prop:freq-back}
  Let\/ $\XX$ be a uniquely ergodic subshift over the finite
  alphabet\/ $\cA$. Let\/ $w$ be any finite word from\/
  $\cA^{*}_{\XX}$ and define\/ $g := 1^{}_{w,0}$ $($so $g (x) = 1$ if
  $w$ occurs in $x$ starting at $0$ and $ g(x) =0$ otherwise$\ts
  )$. Let\/ $\YY \subset \{ 0,1 \}^{\ZZ}$ be the subshift factor
  associated to $g$. Then, the absolute frequency\/ $\nu^{}_{w}$ of\/
  $w \in \cA^{*}_{\XX}$ is determined by the spectral measure\/
  $\sigma^{}_{\! g}$ of\/ $g$, via\/ $\nu^{}_{w} =
  \widehat{\sigma^{}_{\! g}} (0) = \eta^{}_{\YY} (0)$.
\end{prop}
\begin{proof} The claim can be verified by a direct calculation,
\[
   \widehat{\sigma^{}_{\! g}}(0)\, = \int_{\TT} \dd \sigma^{}_{\! g}
   \, = \, \langle g \ts | \ts g \rangle
   \,  = \int_{\XX} |g|^2 \dd \mu \, =
   \int_{\XX} g \dd \mu \, = \, \nu^{}_{w} \, .
\]
Here, the penultimate step relies on $g$ being a characteristic
function, while the last step is an application of the ergodic
theorem, as in Eq.~\eqref{eq:eta-formula}. A comparison with
Proposition~\ref{prop:specmeas} and Eq.~\eqref{eq:gen-eta-coeff} shows
that one also has $ \widehat{\sigma^{}_{\! g}} (0) = \eta^{}_{\YY} (0)$.
\end{proof}

Let now $(\XX,\ZZ)$ be a uniquely ergodic subshift.  The cylinder sets
defined by finite words $w\in \cA^{*}_{\XX}$ form a $\pi$-system of
the Borel $\sigma$-algebra of $\XX$. Consequently, the frequencies of
the finite words uniquely and completely \emph{determine} a shift
invariant probability measure $\mu$ on $\XX$.  If $\XX$ is minimal,
then $\mu$ in turn determines $\XX$ (as $\XX$ is the support of
$\mu$).  In this situation, we call the measure-theoretic, strictly
ergodic subshift $(\XX,\ZZ,\mu)$ \emph{completely reconstructible}
from a collection of measures on $\TT$ if the frequency of any word
can be determined from the Fourier coefficient at $0$ of a suitable
measure from the collection.  Our findings so far can be summarised as
follows.

\begin{theorem}\label{thm:symbolic}
  Let\/ $\XX$ be a uniquely ergodic subshift over the
  finite alphabet\/ $\cA$. Then, the following properties hold.
\begin{enumerate}
\item The fundamental diffraction of any subshift factor of\/
  $\XX$ over a finite\/ $\CC$-valued alphabet is a spectral measure
  of\/ $\XX$.
\item Any spectral measure of the form\/ $\sigma^{}_{\! g}$ with\/ $g$
  from the dense subspace\/ $\AAA (\XX)$ of\/ $L^2 (\XX,\mu)$ arises
  as the fundamental diffraction measure of a subshift factor over a
  finite\/ $\CC$-valued alphabet.
\item If $\XX$ is also minimal, $(\XX,\ZZ,\mu)$ is completely
  reconstructible from the fundamental diffraction measures of the
  collection of subshift factors with $\{0,1\}$-valued alphabets
  $\,($under the assumption that one knows the factor maps as
  well$\,)$.
 \end{enumerate}
\end{theorem}
\begin{proof} As shown in Lemma~\ref{lem:factors}, any subshift factor
  over a finite set emerges from a function $g \in \AAA (\XX)$. Now,
  the first claim follows from Proposition~\ref{prop:specmeas}.

  To prove the second claim, we recall that $\AAA (\XX)$ is dense in
  $C (\XX)$ by the Stone--Weierstrass theorem, and hence also dense in
  $L^2 (\XX,\mu)$.  Then, the remaining part of the claim follows from
  Proposition~\ref{prop:specmeas}.

  As already discussed just before the theorem, any strictly ergodic
  subshift is completely determined by the (positive) frequencies of
  its finite subwords. The corresponding shift invariant measure $\mu$
  is given by its values on the $\pi$-system of cylinder sets defined
  by the finite subwords. Since we assume the knowledge of the factor
  maps, any such frequency can be extracted as the Fourier coefficient
  $\widehat{\sigma^{}_{\! g}} (0)$ with a $\{0,1\}$-valued function
  $g$, as shown in Proposition~\ref{prop:freq-back}. This proves the
  third claim.
\end{proof}

\begin{remark}
  We distinguish $(\XX,\ZZ)$ and $(\XX,\ZZ,\mu)$ at this point, in the
  sense that the knowledge of the former, even if it is known to be
  uniquely ergodic, does not provide the invariant measure
  \emph{explicitly}. Of course, in the uniquely ergodic case, the
  measure $\mu$ is determined via the word frequencies, and the latter
  emerge from uniformly converging limits (averages). However, this
  does not provide their concrete values. Two notable exceptions have
  been studied in the literature, namely shift spaces that are defined
  via primitive substitutions (compare \cite{Q,TAO} and references
  therein), where the frequencies are available via Perron--Frobenius
  theory, and shift spaces that emerge from the projection method (see
  \cite{TAO} for details), where the frequencies are given by certain
  integrals.  The use of a suitable system of factors, as discussed
  above, contains both cases and extends them to a setting that is
  independent of substitutions or projections.
\end{remark}

\begin{remark}\label{rem-2}
  A spectral measure $\sigma$ is called \emph{maximal} if any other
  spectral measure of the same dynamical system is absolutely
  continuous with respect to $\sigma$.  In general, it not true
  (compare \cite{Herning} and Remark~\ref{rem-4})
  that the maximal spectral measure of a subshift can be realised as
  the fundamental diffraction of a subshift factor. However, the
  theorem opens up the possibility to construct a measure equivalent
  to a maximal spectral measure via diffractions of factors. To do
  so, one chooses a countable subset $\mathcal{D}$ of $\AAA (\XX)$,
  which is dense in $C(\XX)$ and hence in $L^2 (\XX,\mu)$.  Now, part
  (2) of Theorem~\ref{thm:symbolic} implies that, for any $f\in
  \mathcal{D}$, the fundamental diffraction $\varrho^{}_{\nts f}$ of the
  subshift factor associated to $f$ is just the spectral measure
  $\sigma^{}_{\! f}$. If $\{ f_n \mid n\in \NN \}$ is an enumeration of
 the elements of $\mathcal{D}$, the measure
\[
   \varrho \, := \sum_{n=1}^{\infty}
   \frac{1}{ 2^n ( 1 + \varrho^{}_{\nts f_n} (\TT))}\, \varrho^{}_{\nts f_n}
\]
is equivalent to the maximal spectral measure, meaning that it has the
same null sets. Indeed, $\varrho$ is absolutely continuous with
respect to any maximal spectral measure as any $\varrho^{}_{\nts f}$
is a spectral measure. Conversely, for any $h \in L^2 (\XX,\mu)$, we
can find a sequence $(h_n)^{}_{n\in\NN}$ in $\mathcal{D}$ that
converges to $h$ (due to the denseness of $\mathcal{D}$).  Then,
$\varrho^{}_{h_n} \! = \sigma^{}_{h_n} $ converges to $\sigma^{}_{\nts
  h}$ in the sense that $\varrho^{}_{h_n} (A) \xrightarrow{\,
  n\to\infty \,} \sigma^{}_{\nts h} (A)$ for any measurable $A\subset
\TT$.  Consequently, $\sigma^{}_{\nts h} $ must be absolutely
continuous with respect to $\varrho$.
\end{remark}

\begin{remark}\label{rem-3}
  It is not hard to see that arbitrarily close to any $g\in \AAA
  (\XX)$ one can find a function $g'$ such that the factor associated
  to $g'$ is actually a \emph{conjugacy}. This means that one can
  construct a spectral measure out of the diffractions of topological
  conjugacies along the lines indicated in Remark~\ref{rem-2}.  In
  particular, the collection of diffractions of all topologically
  conjugate subshifts is then equivalent to the dynamical spectrum of
  the original system. This ties in well with the fact that the
  dynamical spectrum is an invariant under conjugacy, whereas the
  diffraction measure is not. More specifically, this corroborates
  that the `obvious' invariant created from diffraction by collecting
  the diffraction measures of all conjugate systems is indeed
  equivalent to the dynamical spectrum of the initial system.
\end{remark}

\begin{remark}\label{rem-4}
  Some classic subshifts were mentioned in the Introduction, including
  the Thue--Morse chain and its generalisations; compare \cite{BG-TM,GTM}
  and references therein. Other cases include random dimers \cite{BvE}
  or the Rudin--Shapiro chain \cite{TAO}. The unifying property of
  these examples is that one needs just \emph{one} specific factor to
  complete the picture. However, it was recently shown in
  \cite{Herning} that this is not always the case, in the sense that
  there are examples where one really needs to consider infinitely
  many (sliding block) factors, each of them being periodic, to cover
  the entire pure point part of the dynamical spectrum. Only all of
  them together thus replace the knowledge obtainable from Fraczek's
  factor.
\end{remark}

Our exposition of the case of symbolic dynamics has an obvious
extension to block substitutions (or lattice substitutions) in higher
dimensions, where one deals with (uniquely ergodic) dynamical systems
$(\XX,\ZZ^{d},\mu)$ in an analogous way; compare
\cite{Robbie,NF,squiral} and references therein. Since this extension
is straight-forward, we leave the explicit formulation to the
reader. For recent examples, we refer to \cite{squiral,BGG}.

More complex is the situation for Delone dynamical systems, which we
need to describe in a geometric setting.  In particular, we now have
to deal with the continuous translation action of the group $\RR^{d}$
(rather than $\ZZ^{d}$).

\section{Delone dynamical systems}
\label{sec:Delone}

Let $\vL\subset \RR^{d}$ be a point set of finite local complexity
(FLC). By Schlottmann's characterisation \cite[Sec.~2]{Martin}, the
latter property means that $\vL - \vL = \{ x-y \mid x,y \in \vL \}$ is
a locally finite set. For FLC sets, the (continuous) hull is defined
as
\begin{equation}\label{eq:def-hull}
   \XX (\vL) \, := \, \overline{ \{ t+\vL \mid t\in \RR^{d} \} }\ts ,
\end{equation}
where the closure is taken in the local topology. Here, two FLC sets
are $\varepsilon$-close (for small $\varepsilon$ say) when they agree
on a centred ball of radius $1/\varepsilon$, possibly after shifting
one set by an element $t\in B_{\varepsilon} (0)$.  Note that the hull
from Eq.~\eqref{eq:def-hull} is compact as a result of the FLC
property \cite{Martin,BL}. An important subset is given by
\[
   \XX_{0} (\vL) \, := \, \{ \vL' \in \XX (\vL)
   \mid  0 \in \vL' \}\ts ,
\]
which is also known as the \emph{discrete} (or punctured) hull or
transversal. We now assume that $\vL$ is Delone, hence certainly not a
finite set, and that the topological dynamical system
$(\XX(\vL),\RR^{d})$ is uniquely ergodic, with invariant probability
measure $\mu$. This, in turn, induces a unique probability measure
$\mu^{}_{0}$ on $\XX^{}_{0} (\vL)$, which (again by Oxtoby's theorem
\cite{Ox}, see \cite{MR} and \cite{FR} for a general formulation in
the context of Delone sets) is given via the relative patch
frequencies as the measures of the corresponding cylinder sets. Here,
the term `relative' refers to the definition of the frequency per
point of $\vL$, not per unit volume of $\RR^{d}$.  The system is
strictly ergodic (meaning uniquely ergodic and minimal) if and only if
the frequencies of all legal patches exist uniformly and are strictly
positive.

Below, we first approach the factors in a way that is suggested by the
situation in the symbolic case, hence by identifying certain patches
and working with their locator (or repetition) sets. To establish the
connection with diffraction, we will then need some smoothing (via the
convolution with a continuous function of small support), because we
are now working with the translation action of $\RR^{d}$. Viewing
point sets as `equivalent' to measures (via their Dirac comb), we
will be led to a more general (and perhaps also more natural) approach
via measures.  \smallskip

If $K\subset \RR^{d}$ is a compact neighbourhood of $0\in\RR^{d}$, we
call the finite sets of the form $P=(\vL - x)\cap K$, with $x\in \vL$,
the \emph{$K$-clusters} of $\vL$. As they are defined, $K$-clusters
are non-empty, and always contain the point $0$ (as its reference
point, say). This definition avoids certain trivial pathologies
that emerge when the empty cluster is included.

Let $P$ be a $K$-cluster of $\vL$. For any $\vL' \in \XX (\vL)$, the
set of $K$-clusters of $\vL'$ is a subset of the $K$-clusters of
$\vL$, as a consequence of the construction of the hull $\XX
(\vL)$. We may thus define the \emph{locator set}
\[
   T^{}_{K,P} (\vL') \, := \, \{ t \in \RR^{d} \mid
    (\vL' - t)\cap K = P \}
   \, = \, \{ t \in \vL' \mid (\vL' - t)\cap K = P \}
   \, \subset \, \vL' \ts ,
\]
which contains the cluster reference points of all occurrences of $P$
in $\vL'$. Note that the second equality follows from our definition
of a cluster. Clearly, $T^{}_{K,P} (\vL') \subset \vL'$ inherits the
FLC property, though it need not be a Delone set (the Delone
property is guaranteed if $\XX (\vL)$ is minimal). If we now set 
\begin{equation}\label{eq:Y-def}
   \YY\, = \, \YY^{}_{K,P} \, := \,
   \{ T^{}_{K,P} (\vL') \mid \vL' \in \XX (\vL) \} \ts ,
\end{equation}
we obtain a topological factor of $\XX (\vL)$.  This follows from the
observation that the mapping $\vL' \mapsto T^{}_{K,P} (\vL')$ is
continuous in the local topology and commutes with the translation
action of $\RR^{d}$, since $T^{}_{K,P} (t+\vL') = t+ T^{}_{K,P}
(\vL')$. We call any factor of this type a \emph{derived factor} of
$(\XX,\RR^{d},\mu)$, and the collection of all of them the \emph{set
  of derived factors}.  Note that, in our case at hand, $\mu$ induces
a unique invariant probability measure $\mu^{}_{\YY}$ on $\YY$, so
that $(\YY,\RR^{d},\mu^{}_{\YY})$ is also uniquely ergodic, again by
an application of \cite[Prop.~3.11]{DGS}; see also
Proposition~\ref{prop:transfer} below. This setting will later be
generalised beyond (unique) ergodicity in Section~\ref{sec:general}.

In view of this situation, we may employ $\vL$ itself, together with
its image in $\YY$, to analyse the factor and its properties. To this
end, consider the Dirac comb $\omega = \delta^{}_{T_{K,P} (\vL)}$,
which is a translation bounded measure by construction. Its
autocorrelation measure $\gamma^{}_{\omega}$ exists, due to (unique)
ergodicity, and reads $\gamma^{}_{\omega} = \omega \circledast
\widetilde{\omega} = \sum_{z\in \vL - \vL} \eta^{}_{K,P} (z) \ts \delta_{z}$,
with
\begin{equation}\label{eq:eta-KP}
    \eta^{}_{K,P} (z) \, = \, \lim_{R\to\infty}
    \frac{1}{\text{vol} (B_{R} (0))}\,
    \text{card} \bigl( (T^{}_{K,P} (\vL) \cap B_{R} (0))
     \cap (z + T^{}_{K,P} (\vL))\bigr).
\end{equation}
Note that we have used \cite[Lemma~1.2]{Martin} for the derivation of
this expression. In particular, $\gamma^{}_{\omega}$ is a pure point
measure with support in $\vL - \vL$, which is a locally finite subset
of $\RR^{d}$. Moreover, the coefficient $\eta(0)$ is the density of
the set $T^{}_{K,P} (\vL)$, which equals the absolute frequency of the
$K$-cluster $P$ in $\vL$ by construction.  Consequently,
$\eta^{}_{K,P}(0) / \text{dens} (\vL)$ is the relative frequency of
the cluster $P$ within $\vL$. Recall that the diffraction measure
$\widehat{\gamma^{}_{\omega}}$ contains a point (or Dirac) measure at
$0$, whose intensity $I(0)$ satisfies $I(0) =
\bigl(\eta^{}_{K,P}(0)\bigr)^{2}$; see \cite[Cor.~9.1]{TAO}. This
gives access to the relative frequency of $P$.

The situation is thus as follows.  Knowing the diffraction measures of
all derived factors of $(\XX,\RR^{d},\mu)$ means knowing their
autocorrelations. If one also knows the corresponding pairs $(K,P)$,
one can then extract all cluster frequencies, and hence the measure
$\mu^{}_{0}$ on $\XX_{0} (\vL)$. The measure $\mu$ on $\XX (\vL)$ is
uniquely determined from $\mu^{}_{0}$ by standard methods, compare
\cite{TAO} and references therein (in some cases, and for $d=1$ in
particular, this can be seen via the suspension as a special flow,
compare \cite{CFS,EW}). The family of these factors thus permits a
reconstruction of the measure-theoretic dynamical system $(\XX (\vL),
\RR^{d}, \mu)$, and hence its dynamical spectrum, at least in an
abstract sense. In fact, viewing $\vL$ as an example of an $(r,R)$-set
with packing radius $r$ and covering radius $R$, the factor maps
select the list of possible clusters in $\vL$, and the diffraction of
a factor then gives the corresponding cluster frequency.  We have thus
shown the following result.

\begin{prop}
  Let\/ $\vL$ be an FLC Delone set such that its hull\/ $\XX=\XX(\vL)$
  defines a uniquely ergodic dynamical system under the translation
  action of\/ $\RR^{d}$. Then, the cluster frequencies can directly be
  computed from the diffraction measures of all locator sets of finite
  clusters within\/ $\vL$.  Any explicit knowledge of a generating set
  of clusters, together with their frequencies as extracted from the
  diffraction measure of the corresponding factors, explicitly
  specifies the invariant measure and thus the measure-theoretic
  dynamical system\/ $(\XX,\RR^{d},\mu)$.  \qed
\end{prop}

Let us pause to comment on the term `generating' in the above
formulation. It is clear that the set of all clusters suffices, but
that is more than one really needs. A collection of clusters (or
patches, if one restricts to closed balls as compact sets) is called
an \emph{atlas} if it defines the hull $\XX (\vL)$ via the rule that
the latter contains all Delone sets which comply with the atlas (in
the sense that no patch of an element of $\XX (\vL)$ is in violation
of the atlas; see \cite{TAO} for more on this notion). Under certain
circumstances, such an atlas can be finite, in which case $\vL$ is
said to possess \emph{local rules}. The vertex set of the classic
Penrose tiling in the Euclidean plane is a famous aperiodic example of
this situation.\smallskip

To expand on the connection between the hull $\XX (\vL)$ and its
factors, we need a refinement of our arguments in
Section~\ref{sec:symbolic}, as the characteristic function of a
cluster is not continuous on $\XX (\vL)$, wherefore it does not lead
to a complete analogue of a sliding block map as used in
Section~\ref{sec:symbolic}, in the sense that the continuous factor
map from $\XX$ to a derived factor is not a `sliding cluster map'
built from a continuous function on the hull that is defined
locally. Moreover, there is no immediate connection between the
spectral measures of $(\XX,\RR^{d},\mu)$ and the diffraction measures
of derived factors. To establish a connection, we need some
`smoothing' or `regularising' operation, as we will now describe;
compare \cite{Hof,LM} for related ideas.

Let $K\subset \RR^{d}$ be compact, $P$ a $K$-cluster of $\vL$, and
choose a (real-valued) function $\varphi \in C_{\mathsf{c}} (\RR^{d})$
with $\mathrm{supp} (\varphi) \subset B_{r_{\mathsf{p}}} (0)$, where
$r_{\mathsf{p}}$ is the \emph{packing radius} of $\vL$,
\[
    r_{\mathsf{p}} \, = \, \sup \{ r > 0 \mid
    B_{r} (x) \cap B_{r} (y) = \varnothing\,
    \text{ for all distinct } x,y \in \vL \}\ts .
\]
For instance, $\varphi^{}_{\varepsilon} (t) = 1 - \frac{\lvert t
  \rvert}{\varepsilon}$ for $t\in B_{\varepsilon} (0)$ and
$\varphi^{}_{\varepsilon} (t) = 0$ otherwise is a possible choice, with
$\varepsilon < r_{\mathsf{p}}$.  Now, define a function
$\chi^{(\varphi)}_{K,P} \! : \, \XX (\vL) \longrightarrow \RR$ by
\begin{equation}\label{eq:chi-def}
      \chi^{(\varphi)}_{K,P} (\vL') \, = \, \begin{cases}
        \varphi(-t) , & \text{if } (\vL'-t)\cap K = P
            \text{ for some } t \in B_{\varepsilon} (0) , \\
         0 , & \text{otherwise}.  \end{cases}
\end{equation}
Due to the condition on the support of $\varphi$, there is at most one
possible translation $t \in B_{\varepsilon} (0)$ for the occurrence of
$P$, wherefore $\chi^{(\varphi)}_{K,P}$ is indeed well-defined.
Moreover, it is a continuous function on $\XX (\vL)$ by
construction. Note that $\chi^{(\varphi)}_{K,P} (\vL')$ can be
rewritten as
\[
   \chi^{(\varphi)}_{K,P} (\vL') \; =
   \sum_{x\in T^{}_{K,P} (\vL')} \varphi (-x) \; = \,
   \bigl( \varphi * \delta^{}_{T^{}_{K,P} (\vL')} \bigr) (0) \ts .
\]
This, in turn, can be used to define $\chi^{(\varphi)}_{K,P} \! : \,
\XX (\vL) \longrightarrow \CC$ for an \emph{arbitrary} $\varphi \in
C_{\mathsf{c}} (\RR^{d})$. The relevance of this class of functions
emerges from the following completeness result \cite{Martin}.

\begin{prop}\label{prop:chi-complete}
  The linear span of all functions of the form\/
  $\chi^{(\varphi)}_{K,P}$, with\/ $K\subset \RR^{d}$ compact, $P$ a\/
  $K\nts$-clusters of\/ $\vL$ and\/ $\varphi\in C_{\mathsf{c}}
  (\RR^{d})$, is a subalgebra of\/ $C (\XX)$ which is dense with
  respect to the supremum norm.
\end{prop}
\begin{proof}
  For $\vL' \in \XX$, we set $\chi^{(\varphi)}_{K,P} (\vL') = \sum_{ x
    \in T^{}_{K,P} (\vL') } \varphi (- x)$, as motivated above.
  The statement about the denseness is then an immediate consequence
  of \cite[Prop.~2.5]{Martin}. Note that there is no need to deal with
  the empty set in our situation, as our original set $\vL$ is
  Delone.  The proof of Proposition 2.5 in \cite{Martin} (which uses
  the Stone--Weierstrass theorem) also shows that the linear span is
  an algebra, which completes the proof.
\end{proof}

The induced mapping $\alpha^{(\varphi)}_{K,P} \! : \, \RR^{d} \times
\XX (\vL) \longrightarrow \RR$ defined by
\[
   (t,X) \, \longmapsto \, \alpha^{(\varphi)}_{K,P} (t,X)
      = \chi^{(\varphi)}_{K,P} (X\! -t)
\]
is continuous. One can check that
\[
     \alpha^{(\varphi)}_{K,P} (t,X) \, = \, \bigl( \varphi *
     \delta^{}_{T_{K,P} (X)}\bigr) (t) \; =
     \sum_{x\in T^{}_{K,P} (X)} \varphi(t-x) \ts .
\]
The function $\varphi$ acts as a `regularisation', and gives rise to a
`smoothed sliding cluster map' on $\XX (\vL)$ via $X \mapsto \varphi *
\delta^{}_{T_{K,P} (X)}$, the latter now interpreted as a regular,
translation bounded measure. This mapping is continuous and commutes
with the translation action of $\RR^{d}$, so that we obtain a factor
system $(\YY, \RR^{d}, \mu^{}_{\YY})$ that is again uniquely ergodic.
Note that the elements of $\YY$, which approximate derived factors
without being derived themselves, may both be considered as
(absolutely continuous) translation bounded measures and as continuous
functions on $\RR^{d}$. The latter point of view allows us to take
(pointwise) products, which will become useful shortly.

Consider the regular measure $\omega^{}_{\varphi} = \varphi *
\delta^{}_{T_{K,P} (\vL)}$ as representative, and observe the relation
\begin{equation}\label{eq:reg-auto}
     \gamma^{}_{\omega^{}_{\varphi}} \, = \, (\varphi *
      \widetilde{\varphi} ) * (\omega \circledast \widetilde{\omega})
      \, = \, (\varphi * \widetilde{\varphi} )
        * \gamma^{}_{\omega}
\end{equation}
with the measure $\omega = \delta^{}_{T_{K,P} (\vL)}$ from above.
Note that $\gamma^{}_{\omega_{\varphi}}$ is absolutely continuous as a
measure (relative to Lebesgue measure), with a Radon--Nikodym density
that is continuous as a function on $\RR^{d}$.  Moreover,
$\gamma^{}_{\omega^{}_{\varphi}}$ clearly is a (Fourier) transformable
measure, in the sense that the Fourier transform exists and is
again a measure; compare \cite{BF,TAO} for background. Here,
one obtains
\begin{equation}\label{eq:reg-diffract}
     \widehat{\gamma^{}_{\omega^{}_{\varphi}}} \, = \,
     \lvert \widehat{\varphi} \ts \rvert^{2} \,
     \widehat{\gamma^{}_{\omega}}
\end{equation}
by an application of the convolution theorem \cite{BF}.

When we use the tent-shaped function $\varphi =
\varphi^{}_{\varepsilon}$ from above,
$\widehat{\varphi^{}_{\varepsilon}} (0) = (2 \, \pi^{n/2} \,
\varepsilon^{n}) /\Gamma (n/2)$, where $\Gamma$ denotes the gamma
function, is the volume of the cone defined by the graph of the
function $\varphi^{}_{\varepsilon}$ over $\RR^{d}$, so that the value
of $\widehat{\gamma^{}_{\omega}} (\{0\})$, and thus the density of the
set $T^{}_{K,P} (\vL)$, can be calculated from $
\widehat{\gamma^{}_{\omega^{}_{\varphi}}} (\{0\})$.

Recall $\omega^{}_{\varphi} = \varphi * \delta^{}_{T^{}_{K,P} (\vL)}$
and observe that
\begin{equation}\label{eq:Dworkin}
\begin{split}
   \gamma^{}_{\omega^{}_{\varphi}} (t) \, & = \,
   \lim_{r\to\infty} \frac{1}{\mathrm{vol} (B_{r} (0))}
   \int_{B_{r} (0)} \overline{ \omega^{}_{\varphi} (s)}
   \, \omega^{}_{\varphi} (s+t) \dd s \,
    = \int_{\YY} \overline{Y (0)}\, Y(t) \dd \mu^{}_{\YY} (Y) \\[2mm]
   & = \int_{\XX} \overline{\chi^{(\varphi)}_{K,P} (X)} \,
       U^{}_{t} \chi^{(\varphi)}_{K,P} (X) \dd \mu (X) \,
    = \, \langle \chi^{(\varphi)}_{K,P} \ts | \ts
          U^{}_{t}\ts \chi^{(\varphi)}_{K,P} \rangle ,
\end{split}
\end{equation}
which essentially is an application of Dworkin's argument
\cite{Dworkin,Martin,DM} to this situation. The new twist (or
interpretation) is that it appears by linking the original system with
a factor.  Note that, under the second integral, the element $Y\in
\YY$ is interpreted as a continuous function on $\RR^{d}$, so that its
evaluation at a point is well-defined, as mentioned earlier.

Since the (continuous) function  $\gamma^{}_{\omega^{}_{\varphi}} (t)$
is positive definite, Bochner's theorem links it to
a unique positive measure on the dual group via Fourier transform.
In our case, this gives
\begin{equation}\label{eq:gen-spec}
      \gamma^{}_{\omega^{}_{\varphi}} (t) \, =
     \int_{\RR^{d}} e^{2 \pi \ii t x} \dd
     \widehat{\gamma^{}_{\omega^{}_{\varphi}}} (x)
     \, = \int_{\RR^{d}} e^{2 \pi \ii t x} \,
     \lvert \widehat{\varphi} (x) \rvert^{2}
     \dd \widehat{\gamma^{}_{\omega}} (x) \ts ,
\end{equation}
which is the desired connection between the spectral measure of
$\chi^{(\varphi)}_{K,P}$ and the diffraction measures of
$\omega^{}_{\varphi}$ and $\omega$, via a comparison with
Eq.~\eqref{eq:Dworkin}.
\smallskip

As we already saw, the connection between the spectral measure of a
function and the diffraction of a factor is not restricted to
functions $\varphi$ with small support. The latter were chosen above
to establish the connection with the locator sets of clusters and to
highlight the relation to our treatment of the symbolic case in
Section~\ref{sec:symbolic}.  Independently, for any given $\XX = \XX
(\vL)$ with an FLC point set $\vL$ and for any $\varphi \in C_{\mathsf{c}}
(\RR^{d})$, one may directly define the mapping $\chi^{}_{\varphi} \!
: \, \XX (\vL) \longrightarrow \CC$ by $X \mapsto \chi^{}_{\varphi}
(X) = \bigl( \varphi * \delta^{}_{\! X} \bigr) (0)$. Our previous
reasoning around Eqs.~\eqref{eq:Dworkin} and \eqref{eq:gen-spec} can
now be repeated, which leads to the following result.

\begin{prop}\label{prop:gen-spec}
  Let\/ $\vL \subset \RR^{d}$ be an FLC point set such that its
  hull\/ $\XX = \XX (\vL)$ defines a uniquely ergodic dynamical
  system\/ $(\XX, \RR^{d})$ under the action of\/ $\RR^{d}$. For\/
  $\varphi \in C_{\mathsf{c}} (\RR^{d})$, consider the continuous
  function\/ $g^{}_{\varphi} \! : \, \XX \longrightarrow \CC$ defined
  by $X \mapsto g^{}_{\varphi} (X) := \varphi * \widetilde{\varphi} *
  \gamma^{}_{\nts X}$, where\/ $\gamma^{}_{\nts X}$ is the
  autocorrelation of\/ $\delta^{}_{\! X}$. Then, one has
\[
   g^{}_{\varphi} (t) \, = \, \langle \chi^{}_{\varphi} \ts | \ts
   U^{}_{t} \chi^{}_{\varphi} \rangle \, =
   \int_{\RR^{d}} e^{2 \pi \ii ts} \dd \sigma (s)
\]
with the spectral measure\/ $\sigma = \lvert \widehat{\varphi}\ts
\rvert^{2} \, \widehat{\gamma^{}_{\nts X}}$.  \qed
\end{prop}

When $\mathrm{supp} (\varphi) \subset B_{r_{\mathsf{p}}} (0)$, this
result is a special case of our previous situation, with $P$ the
(trivial) singleton cluster and $K= \overline{B_{\varepsilon} (0)}$
for some $\varepsilon < r_{\mathsf{p}}$.  The findings of this section
can now be summarised as follows.

\begin{theorem}\label{thm:Delone}
  Let\/ $\vL$ be an FLC point set with hull\/ $\XX = \XX (\vL)$
  such that\/ $(\XX, \RR^{d})$ is a uniquely ergodic dynamical system,
  with invariant measure\/ $\mu$. Let\/ $K\subset \RR^{d}$ be compact
  and\/ $P$ a $K\nts$-cluster of\/ $\vL$. Then, the following
  properties hold.\smallskip
\begin{enumerate}
 \item The absolute frequency of\/ $P$ in\/ $\vL$ is\/
   $\eta^{}_{K,P} (0)$, where\/ $\eta^{}_{K,P}$ is the autocorrelation
   coefficient from Eq.~\eqref{eq:eta-KP}. Moreover, when\/
   $\gamma^{}_{\omega}$ is the autocorrelation measure of the
   translation bounded measure\/
   $\omega = \delta^{}_{T^{}_{K,P} (\vL)}$, one has\/
   $\widehat{\gamma^{}_{\omega}} \bigl( \{ 0 \} \bigr)=
   \bigl( \eta^{}_{K,P} (0)\bigr)^{2}$.\smallskip
 \item If\/ $\varphi\in C_{\mathsf{c}} (\RR^{d})$, the regularisation\/
   $\omega^{}_{\varphi} = \varphi * \omega = \varphi *
   \delta^{}_{T^{}_{K,P} (\vL)}$ possesses the autocorrelation measure\/
   $\gamma^{}_{\omega^{}_{\varphi}} = \varphi * \widetilde{\varphi} *
   \gamma^{}_{\omega}$ and the diffraction measure\/
   $\widehat{\gamma^{}_{\omega^{}_{\varphi}}} = \lvert \widehat{\varphi}
   \rvert^{2}\, \widehat{\gamma^{}_{\omega}}$. The latter is the
   spectral measure of the continuous function\/ $\chi^{(\varphi)}_{K,P}
   \! : \, \XX \longrightarrow \CC$ defined by\/ $X \mapsto
   \chi^{(\varphi)}_{K,P} (X) = \sum_{x\in T^{}_{K,P} (X)} \varphi (-x)$,
   and one has\/ $\widehat{\gamma^{}_{\omega^{}_{\varphi}}} \bigl( \{ 0 \}
   \bigr) = \lvert \widehat{\varphi} (0) \rvert^{2} \cdot \bigl(
   \eta^{}_{K,P} (0) \bigr)^{2}$.\smallskip
 \item Every spectral measure of\/ $(\XX, \RR^{d}, \mu)$ can be
   approximated arbitrarily well by a finite linear combination of
   diffraction measures of factors that are obtained by smoothed
   sliding cluster maps based on functions of type\/
   $\chi^{(\varphi)}_{K,P}$. In this sense, the diffraction spectra of
   such factors explore the entire dynamical spectrum of\/ $(\XX,
   \RR^{d}, \mu)$.
\end{enumerate}
\end{theorem}
\begin{proof}
  The first claim derives from Eq.~\eqref{eq:eta-KP} and the arguments
  given there, while the connection between
  $\widehat{\gamma^{}_{\omega}} \bigl( \{ 0 \} \bigr)$ and
  $\eta^{}_{K,P} (0)$ is standard; compare \cite[Cor.~9.1]{TAO}.

  The first part of the second claim is a consequence of
  Eq.~\eqref{eq:reg-auto}, which follows from an elementary
  calculation, and Eq.~\eqref{eq:reg-diffract}, which results from an
  application of the convolution theorem to this situation; compare
  \cite[Thm.~8.5]{TAO}. The second part is the combination of
  Eqs.~\eqref{eq:Dworkin} and \eqref{eq:gen-spec}; see also
  Proposition~\ref{prop:gen-spec}, applied to $\vL' = T^{}_{K,P}
  (\vL)$.

  The third claim follows from Proposition~\ref{prop:chi-complete}
  and the observation that the closeness of two continuous functions
  on $\XX$  in the norm topology implies that the corresponding
  spectral measures are close in the vague topology.
\end{proof}

Let us finish this section by formulating a variant of
Theorem~\ref{thm:Delone}.  Let $\vL \subset \RR^{d}$ be an FLC Delone
set, and $(\XX,\RR^{d})$ the associated topological dynamical system.
Recall that any $K$-cluster $P$ of $\vL$ comes with a factor
\[
   \YY^{}_{K,P} \, := \, \{ T^{}_{K,P} (\vL') \mid \vL' \in \XX (\vL) \}\ts ,
\]
which is derived from $\XX$ via $(K,P)$.  If the original system is
uniquely ergodic, then so are all of its factor systems, and all
derived factors in particular. Our previous calculations then show
that the autocorrelation measure of the factor $\XX_{K,P}$ is given as
\begin{equation}\label{eq:gam-om}
   \gamma^{}_{K,P} \, = \, \gamma^{}_{\omega} \, = \,
    \omega \circledast
   \widetilde{\omega} \, = \sum_{z\in \vL - \vL}
   \eta^{}_{K,P} (z) \ts \delta_{z} \ts ,
\end{equation}
with the coefficients $\eta^{}_{K,P}$ from Eq.~\eqref{eq:eta-KP}. We
can now turn Theorem~\ref{thm:Delone} into the following analogue of
Theorem~\ref{thm:symbolic} from Section~\ref{sec:symbolic}.

\begin{coro}
  Let\/ $\vL\subset \RR^{d}$ be an FLC Delone set, and assume that the
  associated dynamical system\/ $(\XX,\RR^{d})$ is uniquely ergodic,
  with invariant measure\/ $\mu$. Then, the following properties hold.
\begin{enumerate}
\item Whenever\/ $\widehat{\gamma}$ is the diffraction measure of a
  derived factor of\/ $(\XX,\RR^{d},\mu)$, the measure $\lvert
  \widehat{\varphi}\rvert^{2} \, \widehat{\gamma}$, with\/ $\varphi
  \in C_{\mathsf{c}} (\RR^d)$ arbitrary, is a spectral measure of\/ $(\XX,
  \RR^d,\mu)$.\smallskip
\item There is a dense set\/ $\cD \subset C(\XX)$, hence also dense
  in\/ $L^{2} (\XX,\mu)$, such that the spectral measure\/
  $\sigma^{}_{\! g}$ of any\/ $g\in \cD$ has the form\/ $\lvert
  \widehat{\varphi}\rvert^{2} \, \widehat{\gamma}$, for some\/
  $\varphi \in C_{\mathsf{c}} (\RR^d)$ and with\/ $\widehat{\gamma}$
  being the diffraction measure of a derived factor of\/
  $(\XX,\RR^{d},\mu)$.  \qed
\end{enumerate}
\end{coro}

Note that, in contract to Theorem~\ref{thm:symbolic}, a derived factor
in this setting in general does not emerge from a factor map of
sliding block or cluster type. Let us also emphasise that, similarly
to Eq.~\eqref{eq:A-def}, the subspace $\cD \subset C(\XX)$ is again
completely explicit, as formulated in
Proposition~\ref{prop:chi-complete}.

\begin{remark}
  In view of these findings, it is suggestive to take a closer look at
  systems with finitely many (non-periodic) factors, up to topological
  conjugacy or up to metric isomorphism. Examples of the former type
  include linearly repetitive FLC systems \cite{Dur,CDP}, while
  Bernoulli shifts provide the paradigm of the latter \cite{Orn}; see
  also \cite[Ch.~7]{Rudolph} for further connections. It is a
  challenge in this context to understand how the diffraction spectra
  of equivalent systems are related. So far, the study of examples
  (compare also our Appendix) suggests that further progress via the
  diffraction approach is indeed possible for systems with finitely
  many factors (up to equivalence).
\end{remark}

Let us now embark on an abstract reformulation in the more general
setting of locally compact Abelian groups.  For convenience, we will
give a self-contained exposition, so that the generalisation of our
previous notions becomes transparent.

\section{An abstract approach}\label{sec:general}

Our considerations will primarily be set in the framework of
topological dynamical systems. We are dealing with $\sigma$-compact
locally compact topological groups and compact spaces.  All
topological spaces are assumed to be Hausdorff.  The general
approach to diffraction via measure dynamical systems discussed below
is largely taken from \cite{BL}; see \cite{BL-2,LS,Lenz} as well.

When $\cX$ is a $\sigma$-compact locally compact space, we denote the
space of continuous functions on $\cX$ by $C(\cX)$, and the subspace
of continuous functions with compact support by $C_{\mathsf{c}}
(\cX)$. The space $C_{\mathsf{c}} (\cX)$ is equipped with the locally
convex limit topology induced by the canonical embeddings $C_K (\cX)
\hookrightarrow C_{\mathsf{c}} (\cX)$, where $C_K (\cX)$ is the space
of complex continuous functions with support in a given compact set
$K\subset \cX$. Here, each $C_K (\cX)$ is equipped with the topology
induced by the standard supremum norm.

As $\cX$ is a topological space, it carries a natural $\sigma$-algebra,
namely the Borel $\sigma$-algebra generated by all closed subsets of
$X$.  The set $\mathcal{M} (\cX)$ of all complex regular Borel measures
on $G$ can then be identified with the space $C_{\mathsf{c}} (\cX)^\ast$
of complex-valued, continuous linear functionals on
$C_{\mathsf{c}} (\cX)$. This is justified by the Riesz--Markov
representation theorem; compare \cite[Ch.~6.5]{Ped} for details. In
particular, we can write $\int_{\cX} f \dd\mu = \mu(f)$ for $f\in
C_{\mathsf{c}}(\cX)$ and simplify the notation this way. The space
$\mathcal{M} (\cX)$ carries the vague topology, which is the weakest
topology that makes all functionals $\mu\mapsto \mu(\varphi)$ on
$\varphi\in C_{\mathsf{c}} (\cX)$ continuous.  The total variation of a
measure $\mu \in \mathcal{M} (\cX)$ is denoted by $|\mu|$. Note that,
unless $\cX$ is compact, an element $\mu \in \mathcal{M} (\cX)$ need
not be bounded. \smallskip

Let $G$ now be a fixed $\sigma$-compact LCAG. The dual group of $G$ is
denoted by $\widehat{G}$, and the pairing between a character
$\widehat{s} \in \widehat{G}$ and $t \in G$ is written as
$(\widehat{s},t)$.  Whenever $G$ acts on the compact Hausdorff space
$\XX$ by a continuous action
\begin{equation*}
   \alpha \! : \; G\times \XX \; \longrightarrow \; \XX
   \, , \quad (t,\omega) \, \mapsto \, \alpha^{}_{t} (\omega) \ts ,
\end{equation*}
where $G\times \XX$ carries the product topology, the pair $\Oomega$
is called a \emph{topological dynamical system} over $G$. We shall
often write $\alpha^{}_{t}\ts \omega$ for $\alpha^{}_{t} (\omega)$,
and think of this as a translation action.  If $\omega\in\XX$
satisfies $\alpha^{}_{t}\ts\omega = \omega$, the element $t\in G$ is
called a \emph{period} of $\omega$. If all $t\in G$ are periods,
$\omega$ is called $G$-invariant, or $\alpha$-{\em invariant\/} to
refer to the action involved.

The set of all Borel probability measures on $\XX$ is denoted by
$\cPO$, and the subset of $\alpha$-invariant probability measures by
$\cPGO$.  An $\alpha$-invariant probability measure is called {\em
  ergodic\/} if every (measurable) $\alpha$-invariant subset of $\XX$
has either measure zero or measure one.  The ergodic measures are
exactly the extremal points of the convex set $\cPGO$. The dynamical
system $\Oomega$ is called \emph{uniquely ergodic} if $\cPGO$ is a
singleton set (which means that it consists of exactly one element).
As usual, $\Oomega$ is called \emph{minimal} if, for all
$\omega\in\XX$, the $G$-orbit $\{\alpha^{}_t\ts \omega \mid t \in G\}$
is dense in $\XX$.  If $\Oomega$ is both uniquely ergodic and minimal,
it is called \emph{strictly ergodic}. \smallskip

Given a $\mu\in \cPGO$, we can form the Hilbert space $\LO$ of square
integrable measurable functions on $\XX$. This space is equipped
with the inner product
\begin{equation*}
    \langle f \ts | \ts g\rangle \, = \,
    \langle f \ts | \ts g\rangle^{}_{\XX} \, :=
    \int_\XX \overline{f(\omega)}\, g(\omega) \dd \mu (\omega).
\end{equation*}
The action $\alpha$ gives rise to a unitary representation $T =
T^\XX := T^{(\XX,\alpha,\mu)}$ of $G$ on $\LO$ by
\begin{equation*}
  T_t \! : \, \LO \; \longrightarrow \; \LO \, ,
  \quad (T_t f) (\omega) \, := \,
  f(\alpha^{}_{-t}\ts \omega) \ts ,
\end{equation*}
for every $f\in \LO$ and arbitrary $t\in G$.

By Stone's theorem, compare \cite[Sec.~36D]{Loomis}, there exists a
projection-valued measure
\begin{equation*}
  E_T\! : \; \{\mbox{Borel sets of $\widehat{G}$}\} \; \longrightarrow \;
  \{\mbox{projections on $\LO$}\}
\end{equation*}
with
\begin{equation*}
  \langle f \ts | \ts T_t f \rangle \, =
  \int_{\widehat{G}} (\widehat{s},t) \dd
  \langle f \ts | \ts E_T(.) f\rangle (\widehat{s}\ts) \, :=
  \int_{\widehat{G}} (\widehat{s}, t) \dd \sigma^{}_{\! f}
  (\widehat{s}\ts)\ts ,
\end{equation*}
where $\sigma^{}_{\! f} = \sigma_f^\XX:= \sigma_f^{(\XX,\alpha,\mu)}$ is
the (positive) measure on $\widehat{G}$ defined by $\sigma^{}_{\! f} (B) :=
\langle f \ts | \ts E_T (B)f\rangle$.  In fact, by Bochner's theorem
\cite{Rudin}, $\sigma^{}_{\! f}$ is the unique measure on
$\widehat{G}$ with $\langle f \ts | \ts T_t f \rangle =
\int_{\widehat{G}}\, (\widehat{s}, t) \dd \sigma^{}_{\! f}
(\widehat{s}\ts)$ for every $t\in G$. The measure $\sigma^{}_{\! f}$ is
called the \textit{spectral measure} of $f$.

The projection-valued measure $E_T$ contains the entire spectral
information on the dynamical system.  It is desirable to encode this
spectral information in terms of measures on $\widehat{G}$. One way of
doing so is via the family of spectral measures. More generally, we
introduce the following definition.

\begin{definition}\label{spectralinvariant} Let $T =
  T^{(\XX,\alpha,\mu)}$ be the unitary representation associated to
  the invariant probability measure $\mu$ on the dynamical system
  $\Oomega$, and let $E_T$ be the corresponding projection-valued
  measure.  A family $ \{ \sigma_\iota \}$ of measures on $\widehat{G}$
  (with $\iota$ in some index set $J$) is called a \emph{complete
    spectral invariant} when $E_T (A) = 0$ holds for a Borel set
  $A\subset \widehat{G}$ if and only if $\sigma_\iota (A) = 0$ holds
  for all $\iota \in J$.
\end{definition}

Let us now turn to factors. Here, we essentially follow the
presentation given in \cite{BL-2}, to which we refer for further
details and proofs.  Let two topological dynamical systems\/ $\Oomega$
and\/ $\Ttheta$ under the action of $G$ and a mapping\/ $\varPhi \! :
\XX \longrightarrow\YY$ be given.  Then, $\Ttheta$ is called a\/ {\em
  factor} of $\Oomega$, with factor map\/ $\varPhi$, if\/ $\varPhi$ is
a continuous surjection with $\varPhi (\alpha^{}_t (\omega)) =
\beta^{}_t (\varPhi (\omega))$ for all\/ $\omega\in \XX$ and\/ $t\in
G$.

In this situation, $\Ttheta$ inherits many features from $\Oomega$.
For example, $U\subset \YY$ is open if and only if\/ $\varPhi^{-1}
(U)$ is open in\/ $\XX$.  Also, $\varPhi$ induces a mapping $\varPhi_*
\! : \, \cMO \longrightarrow \cMT$, $\rho\mapsto\varPhi_*(\rho)$, via
$\big(\varPhi_*(\rho)\big) (g) := \mu(g\circ\varPhi)$ for all $g\in
C(\YY)$. If $\mu$ is a probability measure on $\XX$, its image
$\nu:=\varPhi_*(\mu)$ is a probability measure on $\YY$.  Moreover, if
$\varPhi$ is a factor map, invariance under the group action is
preserved. In fact, $\varPhi_*$ is a continuous surjection from the
set $\cP^{}_{G} (\XX)$ of invariant measures on $\XX$ onto the set
$\cP^{}_{G} (\YY)$ of invariant measures on $\YY$.  Based on the
results of \cite{DGS}, some important properties can be summarised as
follows; see \cite{BL-2} as well.

\begin{prop} \label{prop:transfer}
  Let\/ $\Ttheta$ be a factor of\/ $\Oomega$, with factor map\/
  $\varPhi \! : \, \XX\longrightarrow \YY$.  If the system\/ $\Oomega$
  is ergodic, uniquely ergodic, minimal, or strictly ergodic, the
  analogous property holds for\/ $\Ttheta$ as well.  \qed
\end{prop}

Now, let $\Ttheta$ be a factor of $\Oomega$ with factor map $\varPhi
\! : \XX\longrightarrow \YY$, and let $\mu \in \cPGO$ be fixed. Denote
the induced measure on $\YY$ by $\nu=\varPhi_\ast (\mu)$. Consider the
mapping
\begin{equation}  \label{i-def}
   i^{\,\varPhi} \! : \; \LT \longrightarrow \LO \, , \quad
   f\mapsto f\circ \varPhi \, ,
\end{equation}
and let $p^{}_\varPhi \! : \LO\longrightarrow \LT$ be the adjoint of
$i^{\,\varPhi}$. Then, the maps $i^{\,\varPhi}$ and $p^{}_\varPhi$ are partial
isometries, and $i^{\,\varPhi}$ is an isometric embedding with
\begin{equation*}
   p^{}_\varPhi \circ i^{\,\varPhi} \; = \; {\rm id}^{}_{\LT}
   \quad \mbox{and} \quad
   i^{\,\varPhi} \! \circ p^{}_\varPhi \; = \; P_{i^{\,\varPhi} (\LT)}\ts ,
\end{equation*}
where ${\rm id}^{}_{\LT}$ is the identity on $\LT$ and $P_{i^{\,\varPhi}
  (\LT)}$ is the orthogonal projection of $\LO$ onto $\cV:=i^{\,\varPhi}
(\LT)$.

Given these maps, we can now summarise the relation between the
spectral theory of $T^\XX$ and that of $T^\YY$ as follows; compare
\cite[Thm.~1]{BL-2}.
\begin{theorem} \label{factor-spectralmeasure} 
  Fix some\/ $\mu\in \cP^{}_{G} (\XX)$ and let\/ $\LO$ and\/ $\LT$ be
  the canonical Hilbert spaces attached to the dynamical systems\/
  $\Oomega$ and\/ $\Ttheta$, with factor map\/ $\varPhi$ and\/
  $\nu=\varPhi_\ast (\mu)$.  Then, the partial isometries\/
  $i_{}^{\ts\varPhi}$ and\/ $p^{}_{\varPhi}$ are compatible with the
  unitary representations\/ $T^\XX$ and\/ $T^\YY$ of\/ $G$ on\/ $\LO$
  and\/ $\LT$, in the sense that
\[
     i^{\,\varPhi}  \circ T_t^\YY \, = \,
     T_t^\XX \circ i^{\,\varPhi} \quad \mbox{and}\quad
     T_t^\YY \circ p^{}_\varPhi  \, = \,
     p^{}_\varPhi \circ T_t^\XX
\]
   hold for all\/ $t\in G$.
   Similarly, the spectral families\/ $E_{T^\YY}$ and\/
   $E_{T^\XX}$ satisfy
\[
  i^{\,\varPhi}\circ E_{T^\YY} (\cdot)  \, = \,
   E_{T^\XX} (\cdot)  \circ  i^{\,\varPhi}
   \quad \mbox{and}\quad
   E_{T^\YY} (\cdot)\circ  p^{}_\varPhi
   \; = \;  p^{}_\varPhi \circ E_{T^\XX} (\cdot) \ts .
\]
  The corresponding spectral measures satisfy\/ $\sigma_g^\YY =
  \sigma_{i^{\,\varPhi} (g)}^\XX$ for every\/ $g\in \LT$.  \qed
\end{theorem}

Let us now specify the dynamical systems we are dealing with and
discuss the necessary background from diffraction theory.  The
material is directly taken from \cite{BL}, where the proofs and
further details as well as references to related literature can be
found.

Let $C>0$ and a relatively compact open set $V \subset G$ be given.  A
measure $\omega\in \MM $ is called \emph{$(C,V)$-translation
  bounded\/} if $\sup_{t\in G} |\omega|(t + V) \leq C$.  It is called
{\em translation bounded\/} if there exists a pair $C,V$ so that
$\omega$ is $(C,V)$-translation bounded. The set of all
$(C,V)$-translation bounded measures is denoted by $\MCV$, the set of
all translation bounded measures by $\MTB$. In the vague topology, the
set $\MCV$ is a compact Hausdorff space. There is an obvious action of
$G$ on $\MTB$, again denoted by $\alpha$, given by
\[
   \alpha \! : \; G\times \MTB \; \longrightarrow \; \MTB \, , \quad
   (t,\omega) \, \mapsto \, \alpha^{}_{t}\ts \omega := \, \delta^{}_t *
   \omega \ts .
\]
Restricted to $\MCV$, this action is continuous.  Here, the
convolution of two convolvable measures $\rho, \sigma$ is defined by
\[
  (\rho*\sigma) (\varphi) \, =  \int_{G} \varphi(r + s)
  \dd \rho(r) \dd \sigma(s)
\]
for test functions $\varphi \in C_{\mathsf{c}} (G)$.

\begin{definition}
  $\Oomega$ is called a dynamical system of translation bounded
  measures on\/ $G$\/ $(${\rm TMDS} for short\/$)$ if there exist a
  constant\/ $C >0$ and a relatively compact open set\/ $V\subset G$
  such that\/ $\XX$ is a closed $\alpha$-invariant subset of\/ $\MCV$.
\end{definition}

Having introduced our systems, we can now discuss the necessary pieces
of diffraction theory. Let $\Oomega$ be a TMDS, equipped with an
$\alpha$-invariant measure $\mu\in \cPGO$. We will profit from
the introduction of the mapping
$ N = N^{\XX} \! : \; C_{\mathsf{c}} (G) \longrightarrow C(\XX)$
defined by $\varphi \mapsto N_{\varphi}$ with
\[
    N_\varphi (\omega) \, :=  \int_G \varphi (-s) \dd\omega(s)
    \, = \, \bigl( \varphi * \omega\bigr) (0) \ts .
\]
The mapping $N$ provides a natural way to consider $C_{\mathsf{c}}
(G)$ as a subspace of $L^{2} (\XX,\mu)$ for the given dynamical
system, which is important for our approach.  In particular, we will
need the subspace
\[
   \mathcal{U}^{\XX} \, := \, \mbox{Closure of the linear span of
   $N_\varphi$, with $\varphi \in C_{\mathsf{c}} (G)$, in $\LO$} \ts .
\]

There exists a unique measure $\gamma=\gamma_{\mu}$ on $G$, called the
\textit{autocorrelation measure}, or autocorrelation for short,
of the TMDS,  with
\[
   \gamma (\overline{\varphi} \ast \psi_{\!\_}) \; = \;
   \langle N_\varphi\ts | \ts N_\psi \rangle
\]
for all $\varphi,\psi \in C_{\mathsf{c}}(G)$, where $\psi_{\!\_}(s) :=
\psi(-s)$.  As usual, the convolution $\varphi \ast \psi$ is defined by
$(\varphi \ast \psi) (t) = \int_{G} \varphi (t- s ) \psi (s) \dd s$.

There is an explicit formula for $\gamma$ as follows. Choose an
arbitrary non-negative $\psi \in C_{\mathsf{c}} (G)$ with $\int_{G}
\psi (t) \dd t =1$. Then, we have
 \begin{equation}\label{eq:gen-auto-def}
   \gamma(\varphi) \, = \int_\XX \int_{\RR^d}
   \int_{\RR^d} \varphi(t + s) \, \psi(t)\dd \widetilde{\omega} (s)
   \dd \omega(t)\dd \mu(\omega),
\end{equation}
for every $\varphi \in C_{\mathsf{c}} (G)$, with $\widetilde{\omega}$
as defined in Section~\ref{sec:prelim}.  The measure $\gamma$ is
positive definite, and does \emph{not} depend on the choice of $\psi$;
see \cite{BL} for details.  Therefore, its Fourier transform
$\widehat{\gamma}$ exists and is a positive measure; compare
\cite[Prop.~8.6]{TAO}. It is called the \emph{diffraction measure} of
the TMDS.

\smallskip

\begin{remark}
  Let us emphasise that this concept of an autocorrelation does not
  rely on a local averaging procedure. Instead, it uses an averaging
  along the measure on the dynamical system, also known as an
  \emph{ensemble average}. This has the advantage that we can deal
  with rather general situations. In fact, not even ergodicity of the
  measure on the dynamical system is needed.  However, one then loses
  the connection to the notion of an autocorrelation of an individual
  member of the hull, which may become relevant for applications.
\end{remark}

The crucial connection between the spectral theory of the dynamical
system and the diffraction theory can be expressed in the following
way, as has been discussed in various places.

\begin{prop} \label{prop:diff-spec}
  Let\/ $\Oomega$ be a TMDS over $G$ with invariant probability
  measure\/ $\mu$, and let\/ $\widehat{\gamma}$ be the associated
  TMDS diffraction measure. Then, for every $\varphi \in C_{\mathsf{c}}
  (G)$, the spectral measure of the function\/ $N_\varphi$ is given by
\[
    \sigma^{}_{N_\varphi} \ts = \,
      |\widehat{\varphi}|^2 \, \widehat{\gamma} \ts .
\]
   This spectral measure is the diffraction measure of the factor
   TMDS defined by\/ $\omega \mapsto \varphi * \omega$.
\end{prop}

\begin{proof}
  In the form stated here, the first claim can be found explicitly in
  \cite{BL}; compare \cite{Dworkin,Lenz,LS,LM-2,DM} for related and
  partly even more general versions.

  The second claim follows from the explicit calculations that we have
  used above in the setting of Delone sets, which readily generalises
  to the setting of LCAGs.
\end{proof}

\begin{remark}\label{rem:spec-meas}
  Let us expand on the meaning of Proposition~\ref{prop:diff-spec} for
  $G=\RR^{d}$.  In general, the diffraction measure $\widehat{\gamma}$
  does not assign finite mass to $\RR^{d}$, and thus cannot be a
  spectral measure of $(\XX,\RR^{d},\mu)$. However,
  Proposition~\ref{prop:diff-spec} shows that replacing
  $\widehat{\gamma}$ by $\vert\widehat{\varphi}^{2}\rvert \,
  \widehat{\gamma}$, which reflects a smoothing by convolution on the
  level of the autocorrelation, yields a spectral measure for any
  $\varphi \in C_{\mathsf{c}} (\RR^d)$. In fact, it is possible to
  extend the result to show that, for any non-negative $h\in L^1
  (\widehat{\RR^d}, \widehat{\gamma})$, the measure $h\ts
  \widehat{\gamma}$ is a spectral measure.

  The argument for this extension can be sketched as follows. The
  proposition allows one to show that there is a unique isometric
  linear map
\[
   \varTheta \! : \, L^2(\widehat{\RR^d}, \widehat{\gamma})
   \longrightarrow L^2 (\XX,\mu) \ts ,
\]
mapping $\widehat{\varphi}$ to $N_\varphi$ for any $\varphi \in
C_{\mathsf{c}} (\RR^d)$; compare \cite{DM,LM,LM-2}. This map is not only
isometric, but also intertwines the translation action by $t\in \RR^d$
with the multiplication by $e^{\ii t (\cdot)}$ (as can easily be seen
for $\varphi \in C_{\mathsf{c}} (\RR^d)$ and then follows by
approximation in the general case).  Consider now $g : = \varTheta
\bigl(\sqrt{h}\, \bigr)$, where $\sqrt{h}$ is an $L^{2}$-function, and
the associated spectral measure $\sigma^{}_{\! g}$. Its (inverse)
Fourier transform is then given as $t \mapsto\langle g \ts | \ts T_t g
\rangle$. Now, a short calculation invoking the properties of
$\varTheta$ shows that
\[
   \langle g \ts | \ts T_t g  \rangle  \, = \,
   \big\langle \varTheta (\sqrt{h}\ts) \ts | \ts \varTheta
   (e^{\ii t (\cdot)} \sqrt{h}\ts) \big\rangle \, =
   \int_{\widehat{\RR^d} } e^{\ii t k}\ts
    h (k) \dd\widehat{\gamma} (k) \ts .
\]
Consequently, the (inverse) Fourier transform of $\sigma^{}_{\! g}$
equals that of $h\ts \widehat{\gamma}$, and the desired statement
follows.
\end{remark}

Returning to the general setting, we can describe the \emph{main idea}
behind our subsequent reasoning as
follows. Proposition~\ref{prop:diff-spec} implies that the diffraction
controls the dynamical spectral theory of the subspace
$\mathcal{U}^{\XX}$.  Whenever $\Ttheta$ is a TMDS factor of $\Oomega$
(by which we mean a factor which is also a TMDS), the diffraction of
$\YY$ will control the dynamical spectral theory of the subspace
$\mathcal{U}^{\YY}$ of the factor. Via the factor map, this means that
the diffraction of $\YY$ actually controls the dynamical spectral
theory of the subspace $i^{\,\varPhi} (\mathcal{U}^{\YY})$ of the
original dynamical system. If there are sufficiently many factors,
their diffraction will control the complete dynamical spectral theory.
Here, the concept of `control the dynamical spectrum' is made precise
in the above definition of a complete spectral invariant.  The concept
of `sufficiently many' factors is given a precise meaning as follows.

\smallskip

\begin{definition} Let $\Oomega$ be a TMDS with invariant probability
  measure $\mu$. Let $(\YY_\iota, \beta_\iota)$ with $\iota\in J$ be a
  family of TMDS factors with factor maps $\varPhi_\iota$ and induced
  measures $\nu_{\iota} := (\varPhi_{\iota})_{*} (\mu)$. Then, this
  family is said to be \emph{total} if the linear hull of the spaces
  $i^{\varPhi_\iota} (\mathcal{U}^{\YY_\iota})$, with $\iota\in J$, is
  dense in $\LO$.
\end{definition}

The main result of this section now reads as follows.

\begin{theorem}\label{main-abstract}
  Let\/ $\Oomega$ be a TMDS over\/ $G$ with invariant probability
  measure\/ $\mu$ and corresponding unitary representation\/ $T$.
  Let\/ $(\YY_\iota, \alpha)$, with $\iota\in J$, be a total family of
  TMDS factors equipped with the induced measures\/ $\nu_\iota:=
  (\varPhi_{\iota})_{*} (\mu)$ and associated diffraction measures\/
  $\widehat{\gamma_\iota}$. Then, the measures\/
  $\widehat{\gamma_\iota}$, with $\iota\in J$, constitute a complete
  spectral invariant of\/ $T$.
\end{theorem}
\begin{proof} We have to show that $E_T (A) = 0$ holds for a Borel set
  $A\subset \widehat{G}$ if and only if $\widehat{\gamma_\iota} (A) =
  0$ holds for all $\iota \in J$. For a Borel set $A\subset
  \widehat{G}$ and a function $\varphi \in C_{\mathsf{c}} (G)$, a short
  calculation gives
\[
    \int_{\widehat{G}}  |\widehat{\varphi}|^2 \, 1_A \dd
    \widehat{\gamma_\iota} \, =
    \int_{\widehat{G}}  1_A \dd \sigma_{N_\varphi^{\YY_\iota} }^{\YY_\iota}
  \, = \int_{\widehat{G}}  1_A \dd \sigma_{   i^{\ts\varPhi_\iota}
     ( N_\varphi^{\YY_\iota} ) }^{\, \XX}
  \, = \, \big\langle i^{\ts\varPhi_\iota}  (N_\varphi^{\YY_\iota})\ts | \ts E_T (A)
     \ts i^{\ts\ts\varPhi_\iota}(N_\varphi^{\YY_\iota})\big\rangle .
\]

Here, $1_A$ denotes the characteristic function of $A$. The first
equality is a consequence of Proposition~\ref{prop:diff-spec},
applied to the TMDS $(\YY_\iota, \alpha)$. The second follows from
Theorem~\ref{factor-spectralmeasure}, and the last
from the definition of the spectral measure. \smallskip

Now, by standard reasoning, $\widehat{\gamma_\iota} (A) = 0$ if and
only if $ \int_{\widehat{G}} |\widehat{\varphi}|^2\, 1_A \dd
\widehat{\gamma} =0$ for all $\varphi \in C_{\mathsf{c}} (G)$. Also,
by our assumption of totality, $E_T (A) = 0$ if and only if
$\big\langle i^{\ts\varPhi_\iota} (N_\varphi^{\YY_\iota}) \ts | \ts
E_T (A)\ts i^{\ts\varPhi_\iota}(N_\varphi^{\YY_\iota})\big\rangle =0$
for all $\varphi \in C_{\mathsf{c}} (G)$ and $\iota \in J$.  This
easily gives the desired equivalence.
\end{proof}

We finish this section by briefly indicating how the discussion of
Delone dynamical systems from Section~\ref{sec:Delone} fits into the
present context (and is in fact contained in it). For simplicity, and
as this is the case in the previous section, we restrict our attention
to $G = \RR^d$ (even though all considerations work in the
general case as well).  First of all, let us note that via the map
\[
   \delta \! : \, \{\mbox{FLC sets in $G$}\} \;\longrightarrow\;
   \{\mbox{translation bounded measure on $G\ts $}\}\ts , \quad
   \vL \mapsto \delta^{}_{\!\vL} = \sum_{x\in \vL}
    \delta_x \ts ,
\]
any FLC set can actually be considered as a translation bounded
measure.  In particular, any Delone dynamical system can be considered
as a TMDS, and the theory developed in this section applies.

Let $\vL$ now be an FLC Delone set and $(\XX,\alpha)$ the
associated dynamical system.  Then, any $K$-cluster $P$ of $\vL$
gives rise to a factor
\[
   \YY  \, = \,  \YY_{K,P}
   \, := \, \{T^{}_{K,P} (\vL') \mid \vL' \in \XX (\vL) \} \ts ,
\]
compare Eq.~\eqref{eq:Y-def}, with factor map
\[
  \varPhi \, = \, \varPhi_{K, P}\nts : \; \XX
  \longrightarrow \YY \ts , \quad X \mapsto T^{}_{K,P} (X) \ts .
\]
This factor will be called the factor \textit{derived from
  $(\XX,\alpha)$ via the $K\nts$-cluster $P$ of $\vL$}.  In the
uniquely ergodic case, the autocorrelation $\gamma^{}_{K,P}$ and the
diffraction $\widehat{\gamma^{}_{K,P}}$ of the factor $\YY_{K,P}$ have
been calculated in Section~\ref{sec:Delone}, see
Eq.~\eqref{eq:gam-om}, under the name of $\gamma^{}_\omega$ and
$\widehat{\gamma^{}_\omega}$, respectively.

The $N$-function associated to this factor is given by
\[
   N  =  N^{K,P} \! : \; C_{\mathsf{c}} (G)\longrightarrow C(\YY) \ts ,
   \quad N_\varphi (\varGamma)  = \sum_{x\in \varGamma} \varphi ( -x) \ts .
\]
Thus, the function $N^{K,P} \!\circ \varPhi_{K, P}$ is given by the
formula
\[
   N^{K,P} \!\circ \varPhi_{K, P} ( X) \, =
  \sum_{ x\in T^{}_{K,P} (X) } \varphi (- x) \ts
\]
Note that $N^{K,P} \!\circ \varPhi_{K, P} = \chi_{K,P}^{(\varphi)}$
with $\chi_{K,P}^{(\varphi)}$ as considered above in
Eq.~\eqref{eq:chi-def}.  Proposition~\ref{prop:chi-complete} then
gives the following result, where we need not assume ergodicity.

\begin{prop}\label{prop:totalfam} 
  Let\/ $\vL\subset \RR^{d}$ be an FLC Delone set and\/ $(\XX,\alpha)$
  the associated dynamical system, with\/ $\alpha^{}_{t} X = t +X$,
  and assume that an invariant measure\/ $\mu$ is given. Let\/ $J$ be
  the set of all pairs\/ $(K,P)$ so that\/ $K$ is compact and\/ $P$ is
  a\/ $K$-cluster of\/ $\vL$.  Then, for any\/ $\iota\in J$, the
  factor\/ $(\YY_\iota,\alpha)$ inherits a canonical measure\/
  $\mu_{\iota}$ that is induced by\/ $\mu$, and the family of all such
  factors is total for\/ $(\XX,\alpha,\mu)$.  \qed
\end{prop}

Note that the result depends on the correct choice of the measures
$\mu_{\iota}$ on the factors, which enter the autocorrelations via the
formula from Eq.~\eqref{eq:gen-auto-def}.  As a consequence of
Proposition~\ref{prop:totalfam} and Theorem~\ref{main-abstract}, we
can now generalise the main result of Section~\ref{sec:Delone} as
follows.

\begin{theorem}\label{thm:gen-thm}
  Let\/ $\vL\subset \RR^{d}$ be an FLC Delone set and\/ $(\XX,\alpha)$
  the associated dynamical system, equipped with an invariant
  probability measure\/ $\mu$.  Let\/ $J$ be the set of all pairs
  $(K,P)$ such that\/ $K\subset \RR^{d}$ is compact and\/ $P$ is a\/
  $K$-cluster of\/ $\vL$.  Then, for any\/ $\iota\in J$, the derived
  factor\/ $(\YY_{\iota}, \alpha)$ inherits a canonical probability
  measure\/ $\mu^{}_{\iota}$ that is induced by\/ $\mu$, and the
  family of diffraction measures\/ $\widehat{\gamma_\iota}$, with\/
  $\iota\in J$, is a complete spectral invariant for\/ $(\XX,\alpha)$.
  \qed
\end{theorem}

If the original system in Theorem~\ref{thm:gen-thm} is uniquely
ergodic, then so are its factors. Thus, any of the factors carries a
canonical invariant probability measure that gives rise to an
autocorrelation and hence to a diffraction measure.  Thus, we obtain
the following corollary, which recovers the main result
of Section~\ref{sec:Delone}.

\begin{coro}\label{coro:Delone}
  Let\/ $\vL\subset \RR^{d}$ be an FLC Delone set and\/ $(\XX,\alpha)$
  the associated dynamical system, which we assume to be uniquely
  ergodic.  Let\/ $J$ be the set of all pairs $(K,P)$ such that\/ $K$
  is compact and\/ $P$ is a\/ $K$-cluster of\/ $\vL$.  Then, the
  family of diffraction measures\/ $\widehat{\gamma_\iota}$, with\/
  $\iota\in J$, is a complete spectral invariant for\/ $(\XX,\alpha)$.
  \qed
\end{coro}

In the situation of Corollary~\ref{coro:Delone}, it does not matter
whether one uses the diffraction measure of an element of $\XX$ (as
we did in Section~\ref{sec:Delone}) or that of the dynamical system
(which is our main focus in this section).

It is also possible to embed the situation of
Section~\ref{sec:symbolic} into the abstract setting. Indeed, any
symbolic sequence gives rise to a weighted Dirac comb on $\ZZ$,
where the weights are chosen according to the corresponding symbol in
the sequence. The analogous comment applies to subshifts under the
action of $\ZZ^{d}$. Since the connection between the spectral and the
diffraction measures is more concrete in this case, the approach
described in Section~\ref{sec:symbolic} is ultimately more useful
here.

\begin{remark}
  In the situation of Theorem~\ref{thm:gen-thm}, we can actually
  choose a countable index set for $J$. More precisely, it suffices to
  consider compact sets with are balls around the origin whose radius
  is an integer number. For each such ball, there are then only
  finitely many clusters, due to the FLC assumption. In some
  favourable cases, the set $J$ can be reduced even further, as
  briefly discussed in Section~\ref{sec:Delone}.
\end{remark}

\begin{remark}
  In the case of symbolic dynamics, compare Remark~\ref{rem-3}, one
  can actually restrict oneself to considering topological conjugacies
  rather than (more general) factors.  At this point, we do not know
  whether one can also derive such a stronger statement in more
  general FLC situations.  Moreover, even in the case of symbolic
  dynamics, examples show that these conjugacies can be more
  complicated than the factors needed to obtain the full dynamical
  spectrum. Therefore, in practice, our theorems about the collection
  of factors being sufficient to obtain the full dynamical spectrum
  may be the more viable way to proceed.
\end{remark}

While the main thrust of our work is certainly the situation where the
diffraction is not pure point, some new insights may also be gained
from our considerations in the pure point situation. This is briefly
discussed in the Appendix. \smallskip

In many explicitly treated examples, it turned out that very few
factors (often just one, in fact) were needed to explore the maximal
spectral measure, which is then an interesting alternative to
Fraczek's theorem from \cite{Frac}.  It is thus an obvious question to
search for good sufficient criteria to assess the totality of a family
of factors.

Our general result of Theorem~\ref{thm:gen-thm} neither depends on
ergodicity nor on the FLC property. Nevertheless, it remains to
be analysed how it can concretely be used in the further understanding
of such more general dynamical systems.

\section{Appendix:\ A brief look at pure point diffraction in the
light of factors}

As is well-known (compare our discussion in the Introduction), pure
point diffraction is equivalent to pure point dynamical spectrum. One
key result of \cite{BL} gives the following precise formulation for
the case $G=\RR^{d}$ as follows.

\begin{prop}
  Let\/ $(\XX,\RR^{d},\mu)$ be a TMDS with diffraction measure\/
  $\widehat{\gamma}$. Then, $\widehat{\gamma}$ is a pure point measure
  if and only if\/ $(\XX,\RR^{d},\mu)$ has pure point dynamical
  spectrum. In this case, the group of eigenvalues of\/
  $(\XX,\RR^{d},\mu)$ is the subgroup of\/ $\RR^{d}$ that is generated
  by the Fourier--Bohr spectrum of the autocorrelation, as defined in
  Eq.~\eqref{eq:FB}.  \qed
\end{prop}

This result shows that the spectrum of the dynamical system is
completely determined by the set of pure points or `atoms' of the
diffraction (the Fourier--Bohr spectrum) if one has pure point
diffraction. This does not mean, however, that the diffraction measure
is a complete spectral invariant. More precisely, it can happen that
there are eigenvalues of the system which do not appear in the
Fourier--Bohr spectrum (though our results show that they must appear
in the Fourier--Bohr spectrum of suitable factors).  Such points are
known as \textit{extinctions}; compare \cite{LM,TAO}.  However, even
in the presence of extinctions, it is still possible to construct a
spectral invariant out of the diffraction $\widehat{\gamma}$ as
follows. Choose a strictly positive Schwartz function $h$ on $\RR^{d}$
such that
\[
   \nu^{}_h  \, := \, h \ts \widehat{\gamma}
\]
is a probability measure on $\RR^{d}$. Such a choice is always
possible, as $\widehat{\gamma}$ is a translation bounded measure by
general principles \cite{BF,TAO}. We note in passing that
$\nu^{}_{\nts h}$ is a spectral measure of $(\XX,\RR^{d},\mu)$ by
Remark~\ref{rem:spec-meas}. Consider now the $n$-fold convolution
 \[
    \nu_h^{*n} \, := \, \nu^{}_h  \ast \cdots
    \ast   \nu^{}_h
\]
of $\nu^{}_h$ with itself for any natural number $n$, where
$\nu_{h}^{*1} = \nu^{}_{h}$. Then, Theorem~\ref{thm:gen-thm}
together with some basic facts on convolutions shows that the family
$\bigl\{ \nu_h^{*n} \bigr\}_{n\in \NN}$ is a complete spectral
invariant.

In some cases, it is known that there exists an natural number $N$
such that any eigenvalue can be expressed as a sum of no more than $N$
elements of the Fourier--Bohr spectrum; see \cite{LM} for a detailed
discussion of this phenomenon. In this case, when $0$ is an element of
the Fourier--Bohr spectrum (which is always true for the
autocorrelation of the standard Dirac comb of a Delone set),
$\nu_h^{*N}$ alone forms a complete spectral invariant.

{}For a TMDS that is based on an FLC Delone set, one can obtain
further information on the extinctions. Our above results, and
Theorem~\ref{thm:gen-thm} in particular, show that the diffraction
measures of all derived factors (all being FLC point sets here) form a
complete spectral invariant. We may conclude that, for any extinction
point of the original system, there is an FLC point set factor which
covers this extinction via its Fourier--Bohr spectrum. This is
interesting for pure point diffractive Delone sets that are known to
possess no non-trivial Delone factors (up to topological conjugacy),
except (possibly) periodic ones \cite{CDP}. In such a situation, all
extinctions are related to periodic factors (if any exists at all) or
to topologically conjugate point sets. Here, one has to bear in mind
that the Fourier--Bohr spectrum is \emph{not} an invariant of
topological conjugacy, while the dynamical spectrum clearly is.
\smallskip

An interesting example is provided by the silver mean point set
$\varLambda$ in its formulation as a regular model set; compare
\cite{BL-2} as well as \cite[Sec.~9.3]{TAO}. It gives rise to a
uniquely ergodic dynamical system.  When the elementary distances are
$1$ (short) and $\lambda = 1 + \sqrt{2}\,$ (long), the diffraction
measure of the canonical Dirac comb with point masses of weight $1$ on
every point of $\varLambda$ reads
\[
    \widehat{\gamma^{}_{\varLambda}} \, = \,
    \sum_{k\in {L}_{}^{\circledast}} I^{}_{\!\varLambda}(k) \, \delta^{}_{k} \ts ,
\]
where the Fourier module $L_{}^{\circledast} = \frac{\sqrt{2}}{4}\,
\ZZ[\sqrt{2}\,]$ coincides with the dynamical spectrum (which is
pure point in this case), while the extinction set is
\begin{equation}\label{eq:sext}
   S_{\mathrm{ext}} \, = \, 
   \{ k \in L_{}^{\circledast} \mid I^{}_{\!\varLambda}(k) = 0 \}
    \, = \, \{ k \in L_{}^{\circledast} \mid k^{\star} =
    \tfrac{m}{\sqrt{2}} \text{ for some } 0 \ne k \in \ZZ \ts \} \ts ,
\end{equation}
see \cite[Rem.~9.10]{TAO} for the details. Clearly, the Fourier
module satisfies $\lambda L_{}^{\circledast} = L_{}^{\circledast}$.

Now, keeping all points from $\varLambda$ at the beginning of a long
interval (and deleting the others) defines a derived factor, as
this map works for the continuous hulls and commutes with
translation. The resulting point sets are silver mean chains
on a larger scale, obtained from the original one by multiplication
with $\lambda$, which simply reflects the local inflation symmetry
of this aperiodic example. Consequently, the factor is actually
locally equivalent (MLD; see \cite{TAO} for details) and thus
topologically conjugate. Interestingly, since we have
\[
    I^{}_{\nts\lambda \varLambda}(k) \, = \,
    I^{}_{\!\varLambda} (\lambda k)\ts ,
\]
the new intensities are non-zero on all points of the original set
$S_{\mathrm{ext}}$ from Eq.~\eqref{eq:sext}, while the extinctions are
now located at the set $S_{\mathrm{ext}}/\lambda$, which had no
extinctions for the original diffraction measure. So, the two
diffraction measures constitute a complete spectral invariant in this
case.

Alternatively, one may view the previous example as a weighted model
set, with different weights for points at the beginning of short or
long intervals. A generic choice of these weights will lead
to a diffraction measure without any extinctions on $L_{}^{\circledast}$,
which is then a complete spectral invariant.

\section*{Acknowledgements}

It is a pleasure to thank Uwe Grimm, E.\ Arthur Robinson and Tom Ward
for helpful discussions, and David Damanik and an anonymous reviewer
for useful hints on the manuscript. This work was supported by the
German Research Foundation (DFG), within the CRC 701.

\end{document}